\documentclass{article}
\usepackage{lineno,hyperref}
\usepackage[a4paper, top=2cm, bottom=2cm, left=2cm, right=2cm]{geometry}

\usepackage{amsfonts}
\usepackage{amssymb}
\usepackage{amsmath}
\usepackage{fancyhdr}
\usepackage{bm}
\usepackage{color}
\usepackage{xcolor}
\usepackage{graphicx}
\usepackage{psfrag}
\usepackage{indentfirst}
\usepackage{latexsym}
\usepackage{mathrsfs}
\usepackage{array}
\usepackage[justification=centering]{caption}
\usepackage{verbatim}
\usepackage[tight, footnotesize]{subfigure}
\usepackage{rotating}
\usepackage{multirow}
\usepackage{pgfplots}
\usepackage{tikz}
\usepackage{graphics} 
\usepackage{epsfig} 

\definecolor{refred}{RGB}{255,0,0}
\definecolor{citeblue}{RGB}{0,0,255}
\hypersetup{
	colorlinks=true,
	linkcolor=refred
}
\hypersetup{
	colorlinks=true,
	citecolor=citeblue
}

\allowdisplaybreaks
\def\leaderfill{\leaders\hbox to 0.3 cm{\hss$\cdot$\hss}\hfill}
\makeatletter
\let\oldsection\section
\renewcommand\section{\setcounter{equation}{0}\oldsection}

\newtheorem{Theorem}{\indent  Theorem}

\newtheorem{Lemma}[Theorem]{\indent  Lemma}
\newtheorem{Definition}{\indent  Definition}

\newtheorem{Assumption}{\indent  Assumption}

\newtheorem{Remark}[Theorem]{\indent  Remark}

\numberwithin{equation}{section}

\title{ \bf The high-order approximation of SPDEs with multiplicative noise via amplitude equations}
\author{ Shiduo Qu\\ School of Mathematics, Southeast University, Nanjing 211189, PR China\\
	\texttt{qusdjlu@hotmail.com}
	\and Hongjun Gao\\ School of Mathematics, Southeast University, Nanjing 211189, PR China\\
	\texttt{hjgao@seu.edu.cn}
}
\date{}
\begin{document}
	\Large
		\maketitle
	\vspace{-10mm}
	\noindent\rule{\linewidth}{0.4pt}
	\noindent {\bf Abstract}
The emphasis of this paper is to investigate the high-order approximation of a class of SPDEs with cubic nonlinearity driven by multiplicative noise with the help of the amplitude equations. The highlight of our work is that we improve the convergence rate between the real solutions and the approximate ones. Precisely, previous results often focused on deriving the approximate solutions via the first-order amplitude equations. However,  the approximate solutions  are constructed by the first-order amplitude equations and the second-order ones in this paper.  And, we rigorously prove that such approximate solutions  enjoy improved convergence property. In order to illustrate this demonstration more intuitively, we apply our main theorem to stochastic Allen-Cahn equation, and provide numerical analysis.

			
		\vspace{-10mm}
		\noindent\rule{\linewidth}{0.4pt}
			\section{Introduction}
		In the wake of developments in science and technology, there is a consensus that multi-scale phenomena are important characteristics of complex systems. The reasons for the occurrence of multi-scale phenomena can include perturbations from small-scale external forces, coupling between fast and slow scales, nonlinear interactions between macroscopic and microscopic patterns, and so on. For multi-scale models, it is usually difficult to obtain explicit solutions or analytical expressions due to the presence of small parameter. A typical example is that if the Reynolds number of Navier-Stokes equation for incompressible flows is too large,  boundary layers appear and  turbulent flows exhibit vortices over a large range of scales, which may cause more complexity \cite{Ewn}. Whereas, it is important to develop effective methods to describe and analyze the asymptotic dynamics of multi-scale systems.
		
		Perturbation theory is a major research field of multi-scale problems.  The origin of perturbation theory can be traced back to the work of astronomers, such as Lindstedt, Bohlin, Gylden, etc. They utilized the power series of small parameter to study the planetary motion and develop techniques to eliminate the secular terms in asymptotic solutions.  Poincar$\acute{\textup{e}}$ further proved that such  parameter expansion method is effective. Thereafter, mathematicians develop various methods to investigate  the approximate solutions of enormous types of perturbation systems, such as matched asymptotic expansions, averaging method and renormalization group method, and so on. The key idea behind these methods is to expand the solution of the original equation in power series with respect to the small parameter.  The coefficients of the same degree are then calculated to construct an approximate solution, followed by the rigorous error estimations. For more backgrounds and details about perturbation theory, please refer \cite{hol,nay}.  Generally, physicians and engineers are often not satisfied with obtaining only the first-order approximate solutions, but instead focus on the higher-order approximate solutions which enjoy higher accuracy and is more helpful in many practical applications, such as aerospace technology, hydrodynamics, thermodynamics and so on.  Motivated by the reason, mathematicians further develop perturbation theory to satisfy the actual requirements. 	
		
		In 1950s, Tikhonov \cite{Tik1,Tik2,Tik3} proved that  the solutions for slow-fast ordinary differential equations can converge the solutions of corresponding degenerate equations, while the convergent rate was unknown. To improve
		the accuracy of Tikhonov's works,  Vasil'eva and Butuzov \cite{Vas1} constructed the high-order approximate solutions by analyzing the boundary layer terms and regular ones of formal asymptotic expansions. Their method is called by boundary layer function method, by which  there are  many developments about  high-order approximations during last decades \cite{Nef1,Nef2,Ni1,Vol}.  Scholars  also  made progress in the researches of  renormalization
		group method. In 1999, Ziane \cite{Zia} proposed  a class of singularly perturbed systems, and obtained the first-order approximate solution by renormalization group method.  In 2008, Lee DeVille  et al \cite{Lee} further obtained the approximate solution  to second order. Recently, Li and Shi \cite{Liw} generalized previous results to any order by developing a systematic and explicit strategy of renormalization group method. For the high-order approximation of PDEs,  Wu and Ni \cite{Wu1,Wu2} studied the high-order approximation of  reaction-diffusion equations with discontinuous reactive term, which can characterize the phenomena on various media interfaces.  Levashova et al  \cite{Lev1} considered   a class of singularly perturbed elliptic equations, which can model  layered semiconductor structures. They proved the existence of smooth solution, and show the accuracy of the asymptotic approximation.
		Chekroun and Liu \cite{Chek1} developed higher-order finite-horizon parameterizing manifolds, and considered the pullback limit of  two-layer backward–forward system of Burgers-type equation.
		And, relying on a power series expansion of the manifold
		function, they derived higher-order approximations of invariant manifolds in \cite{Chek4}.

		
		In addition to multi-scale phenomena, uncertainty and random fluctuations are also widespread in  the natural world. In  the recent two decades, more attentions are paid to  asymptotic behaviors of  multi-scale stochastic differential equations.
		So far, many methods are used to   investigate effective dynamics of SDEs, such as averaging method \cite{cer1,cer2}, homogenization method \cite{Wang1,Wang2}, slow manifolds \cite{Fu1,schm}, and so on.
		In particular, there are also some investigations about the high-order approximation in this field.  Birrell and  Wehr \cite{Bir} obtained the high-order approximation of stochastic Langevin equations via the derivation of a hierarchy of SDEs. Following the idea from \cite{Bir}, Li et al \cite{Lix} used averaging method to derive the high-order approximate solutions of a class of stochastic slow-fast systems. From the view of qualitative analysis, Chekroun et al \cite{Chek2,Chek3} developed stochastic parameterizing manifolds, showed  the general framework of  n-layer backward-forward system. Based the results in \cite{Chek2,Chek3}, Li et al \cite{Li1},  and Xiao and Gao \cite{Xiao1,Xiao2} used  2-layer backward-forward system to investigate stochastic attractor bifurcation of Swift-Hohenberg equation.  For smooth stable invariant manifold, Guo and Duan \cite{Gu3} taked the Taylor expansion of the nonlinear operator at the point of zero, and derived the high-order approximation of the local geometric shape.  The previous results are devoted to SDEs or SPDEs with linear multiplicative noise, while few developments are achieved about the high-order approximation fo SPDEs with nonlinear multiplicative noise up to now. In addition, since it is still an open problem that whether SPDEs with nonlinear multiplicative noise can generate random dynamical system, the results in \cite{Chek2,Chek3,Gu3} are not applicable to nonlinear multiplicative noise. Whereas, the present work is to use another approach, amplitude equations,  to investigate the high-order approximate solutions of multi-scale SPDEs with nonlinear multiplicative noise.

		Amplitude amplitudes are the modes, which can  dominate the evolution behaviors of complex systems.
		For SPDEs, the advantages of amplitude equations is not only to construct  the approximate solutions of SPDEs, but also analyze whether the noise  might shift the bifurcation. In 2001,  Bl$\ddot{\textup{o}}$mker et al used amplitude equations to investigate the attractivity and approximation for stochastic Swift-Hohenberg equation with additive noise on bounded domain, which is the
		first rigorous result in this field. After that,  there are many extensions for the case that additive noise and bounded domain \cite{Bl5,Bl1,Bl3,Bl4}. More general cases are concerned in the last few years, such as multiplicative noise \cite{Bl6,Fu1}, non-Markovian noise \cite{Bl8,Nea,yuan1}, and large or unbounded domains \cite{Bi1,Bi2,Bl0,Mo6}.  Moreover, amplitude equations also  contribute to the researches of the random dynamics, such as  invariant measure \cite{Bl5},  random attractors \cite{Bl13}, and center manifolds \cite{Bl7}.  For the investigations of high-order amplitude equations, Bl$\ddot{\textup{o}}$mker and Mohammed in \cite{Bl3} obtained the second-order amplitude equations for SPDEs with additive noise. However, to the best of our knowledge, few results are developed for  multiplicative noise.

		
		In this paper, we consider the following SPDEs  with multiplicative noise:
		\begin{align}
			\textup{d}u(t)&=\mathcal{A}u(t)+\varepsilon^{2}\mathcal{L}u(t)+\mathcal{F}(u(t))\textup{d}t+\varepsilon G(u(t))\textup{d}W(t),~~u\in\mathcal{H},\label{eq10}
		\end{align}
		where $\mathcal{H}$ is an infinite dimensional separate Hilbert space with
		scalar product $\langle\cdot,\cdot\rangle$ and corresponding norm $\|\cdot\|$, $\mathcal{A}$ is assumed to be a self-adjoint and non-positive operator with finite dimensional kernel space called dominant modes,  $\mathcal{L}$ is
		a linear operator, $\mathcal{F}$ is cubic nonlinearity with dissipativity condition, $G(u(t))$ is a Hilbert-Schmidt operator with $G(0)= 0$, $W(t)$ is a standard cylindrical Wiener process on some stochastic basis. In (\ref{eq10}), $\varepsilon$ is the small parameter
		representing the distance from bifurcation in the drift part,  while it characterizes the strength of the noise in the diffusion one.
		It is worthwhile to note that many models fit in the framework of (\ref{eq10}), such as stochastic Allen-Cahn equation, stochastic Swift-Hohenberg equation.
		From the perspective of perturbation theory, (\ref{eq10}) is seen as   PDEs  with small deterministic perturbation and stochastic force.
		In \cite{Fu1}, Bl$\ddot{\textup{o}}$mker and Fu derived the first-order amplitude equations of (\ref{eq10}), used it to construct the approximation solutions of (\ref{eq10}), and proved that the error between  the real solutions and the approximate ones is of order $\varepsilon^2$. Indeed, such convergence rate is comparable to the order of linear perturbation $\mathcal{L}u$. As mentioned earlier,  physicists and engineers are interesting in pursing more precise  and valid approximation solutions of perturbation systems   to meet the demands in accuracy of physical experiments and practical applications. Therefore, it is meaningful and necessary to further investigate the high-order approximation of  (\ref{eq10}). In order to achieve this goal, we want to abstract more effective information from dominant modes. In other words, we not only need the first-order amplitude equations, but also the second-order ones. 
		
		The aim of this paper is to derive the  first-order and second-order amplitude equations of (\ref{eq10}), and use them to obtain the high-order approximate solutions. In brief,  (\ref{eq10})  is firstly decomposed into slow modes and fast ones after time scaling. Secondly, we abstract the first-order amplitude equations from slow modes, and use it to obtain  the high-order approximation of fast modes.  Finally, we further derive the second-order amplitude equations, and
		construct the high-order approximate solutions of (\ref{eq10}). The main difficulties of this paper happen in the last step. Precisely, the first-order amplitude equations can be obtained by separating the terms of order $\varepsilon$  from slow modes as \cite{Fu1} . However,
		in order to obtain the second-order amplitude equations, we have to eliminate those terms of order $\varepsilon^{2}$ hiding in slow modes. Moreover,  the  terms of $\varepsilon^{2}$  can be separated  after careful error estimations, but it exist still the fast fluctuations in the diffusion part.
		Indeed, if such fast fluctuations are  $0$, then we can obtain the second-order amplitude equations, and prove the approximate solutions converge to the real ones with convergence rate $\varepsilon^{3}$.  If else, we need further  consideration on  the case that such fast fluctuations are not $0$. Similar problems were solved in \cite{Bl1,Bl3,Mo3}, in which the fluctuating fast OU-processes appear in the diffusion part. However, the fast fluctuations that we encounter are more general than OU-processes since  they are coupled with the first amplitude equations. Therefore, we need more precise error analysis to overcome this trouble, and obtain the second amplitude equations as well as the approximate solutions. From the technical reason, we just can  show  the approximate solutions weakly converge the real ones without explicit rate when the kernel space of $\mathcal{A}$ has more than one dimension. But, if such kernel space is one-dimensional, we can rigorously prove that the error rate is of order  $\varepsilon^{\frac{5}{2}}$.
		
		The rest of the paper is organized as follows. In Section 2,  we list basic assumptions and review  relevant concepts. In Section 3, the first-order amplitude equations is shown, and the  error analysis about the slow and fast  modes are given.
		In Section 4,  we present the second-order amplitude equations and  main theorems. In Section 5, we apply our results to stochastic Allen-Cahn  equation.
		The last Section is conclusion and future research interest.
		\section{Assumptions and Notations}
		Referring to the framework of \cite{Fu1,Bl3,Mo3}, some assumptions and notations we use later are presented here. In the sequel, we take C as a positive constant, which may be different between
		occurrences but does not depend on  the parameter $\varepsilon$ unless otherwise specified.
		\begin{Assumption}\label{assu1}
			Assume that $\mathcal{A}$ is a non-positive and self-adjoint operator on $\mathcal{H}$ with eigenvalues $0=\lambda_{1}\leq\cdot\cdot\cdot\leq\lambda_{k}\cdot\cdot\cdot$,
			and $\lambda_{k}\geq Ck^{m}$ holds for all sufficiently large $k$, positive constants $m$ and $C$. Suppose that there is a complete orthonormal basis $\{e_{k}\}_{k\in\mathbb{N}}$ such that $\mathcal{A}e_{k}=-\lambda_{k}e_{k}$. Denote $\ker\mathcal{A}$ and the orthogonal complement
			of it by $\mathcal{N}$ and $\mathcal{S}$. Suppose  $\dim\mathcal{N}=N<\infty$.
		\end{Assumption}
		
		Define projections $P_{c}: \mathcal{H}\rightarrow \mathcal{N}$ and $P_{s}=\textup{I}-P_{c}$ with identical operator $\textup{I}$.
		For a map $L$, we use $L_{c}:=P_{c}L$ and $L_{s}:=P_{s}L$.
		\begin{Definition}
			For $\alpha\in\mathbb{R}$, we define the space $\mathcal{H}^{\alpha}$ as
			\begin{equation*}
				\mathcal{H}^{\alpha}:=\mathcal{D}(\textup{I}-\mathcal{A})^{\alpha})=\Big\{\sum^{\infty}_{k=1}\gamma_{k}e_{k}:\sum^{\infty}_{k=1}\gamma_{k}^{2}(\lambda_{k}+1)^{2\alpha}<\infty\Big\}
			\end{equation*}
			with the norm
			\begin{equation*}
				\|\sum^{\infty}_{k=1}\gamma_{k}e_{k}\|_{\alpha}=\Big(\sum^{\infty}_{k=1}\gamma_{k}^{2}
				(\lambda_{k}+1)^{2\alpha}
				\Big)^{\frac{1}{2}}.
			\end{equation*}
		\end{Definition}
		
		
		\par If $\mathcal{A}$ satisfies Assumption $\ref{assu1}$, it generates an analytic semi-group $\{e^{\mathcal{A}t}\}_{t\geq 0}$ on any space $\mathcal{H}^{\alpha}$, defined by
		\begin{equation*}
			e^{\mathcal{A}t}(\sum^{\infty}_{k=1}\gamma_{k}e_{k})=\sum^{\infty}_{k=1}e^{-\lambda_{k}t}\gamma_{k}e_{k},~~ t\geq0.
		\end{equation*}
		Another property of $\mathcal{A}$ is the following inequality.
		\begin{Lemma}\textup{\cite{Pz}}\label{l51}
			Under Assumption $\ref{assu1}$, for all $\beta\leq\alpha$, $\rho\in(\lambda_{n},\lambda_{n+1}]$, there is a constant $M>0$,
			such that for all $u\in\mathcal{H}^{\beta}$,
			\begin{equation*}
				\|e^{\mathcal{A}t}P_{s}u\|_{\alpha}\leq Mt^{-\frac{\alpha-    \beta}{m}}e^{-\rho t}\|P_{s}u\|_{\beta},~~\forall t\geq 0.
			\end{equation*}
		\end{Lemma}
		\begin{Assumption}\label{assu2}
			Assume that $\mathcal{L}:\mathcal{H}^{\alpha}\rightarrow\mathcal{H}^{\alpha-\beta}$ is a linear continuous mapping that commutes with $P_{c}$ and $P_{s}$, for some $\alpha\in\mathbb{R}$ and $\beta\in [0,m)$.
		\end{Assumption}
		
		We remark that it is sufficient for the above assumptions that $\mathcal{A}$ is polynomial of the Laplacian and $\mathcal{L}$ is a lower order differential operator commuting with $\mathcal{A}$.
		\begin{Assumption}\label{assu3}
			Let $\alpha$ and $\beta$ be given in Assumption $\ref{assu2}$. Assume that $\mathcal{F}:(\mathcal{H}^{\alpha})^{3}\rightarrow\mathcal{H}^{\alpha-\beta}$ is a trilinear, symmetric mapping that satisfies the following conditions:
			\begin{equation}\label{eq41}
				\|\mathcal{F}(u,v,w)\|_{\alpha-\beta}\leq C_{0}\|u\|_{\alpha}\|v\|_{\alpha}\|w\|_{\alpha},~~ \forall u,v,w\in\mathcal{H}^{\alpha},
			\end{equation}
			and $\forall u,v,w\in\mathcal{N}$,
			\begin{align}
				&\langle\mathcal{F}_{c}(u,u,w),w\rangle\leq0,\label{eq43}\\
				&\langle\mathcal{F}_{c}(u+v+w)-\mathcal{F}_{c}(v),u\rangle\leq -C_{2}\|u\|^{4}+C_{3}\|w\|^{4}+C_{4}\|w\|^{2}\|v\|^{2},\label{eq44}
			\end{align}
			where $C_{0}, C_{1},C_{2},C_{3},C_{4}$ are positive constants, and $\mathcal{F}(u):=\mathcal{F}(u,u,u)$.
		\end{Assumption}
		
		\begin{Assumption}\label{assu4}
			Let $U$ be a separable Hilbert space with inner product $\langle\cdot,\cdot\rangle_{U}$. Assume that $W(t)$ be a standard cylindrical Wiener process in $U$ on a stochastic base
			$(\Omega,\mathscr{F},\{\mathscr{F}_{t}\}_{t\geq 0},\mathbb{P})$.
		\end{Assumption}
		\par Note that $W(t)$ has the expansion \cite{Da}
		\begin{equation*}
			W(t)=\sum^{\infty}_{k=1}\beta_{k}(t)f_{k},
		\end{equation*}
		where $\{\beta_{k}(t)\}_{k\in\mathbb{N}}$ are real valued Brownian motions mutually independent on the above stochastic basis, and $\{f_{k}\}_{k\in\mathbb{N}}$ is a complete orthonormal system in $U$.
		
		We remark $W(t)$ is not $U$-valued random process, but the stochastic integral is well defined in this paper. For stochastic integral for cylindrical Wiener processes, please see \cite{Da}.
		
		Here we recall the definition of a Hilbert-Schmidt operator. Suppose that $U_{1}$, $U_{2}$ are two separable Hilbert spaces. A linear and bounded operator $L:U_{1}\rightarrow U_{2}$ is said to be Hilbert-Schmidt if
		\begin{equation*}
			\sum^{\infty}_{j=1}\|L l_{j}\|^{2}_{U_{2}}<\infty,
		\end{equation*}
		where $\{l_{j}\}_{j\in\mathbb{N}}$
		is a complete orthonormal basis in $U_{1}$. We remark the norm is
		independent of the choice of the basis $\{l_{j}\}_{j\in\mathbb{N}}$.

		Denote the set of all Hilbert-Schmidt operators from $U_{1}$ to $U_{2}$ by $\mathscr{L}_{2}(U_{1},U_{2})$. Note that $\mathscr{L}_{2}(U_{1},U_{2})$
		is separable Hilbert space, with the scalar product
		\begin{equation*}
			\langle\cdot,\cdot\rangle_{\mathscr{L}_{2}(U_{1},U_{2})}=\sum^{\infty}_{j=1}\langle \cdot l_{j}, \cdot l_{j} \rangle_{U_{2}},
		\end{equation*}
		by which the norm of $\mathscr{L}_{2}(U_{1},U_{2})$ is induced.
		\begin{Assumption}\label{assu5}
			Suppose that $G:\mathcal{H}^{\alpha}\rightarrow\mathscr{L}_{2}(U,\mathcal{H}^{\alpha})$ with $\alpha$ as in Assumption $\ref{assu2}$, and $G(0)=0$. Moreover, $G(\cdot)$ is
			is Fr$\acute{\textsl{e}}$chet differentiable up to order 3, and satisfies $\forall u,v,w,x\in\mathcal{H}^{\alpha}$ with $\|u\|_{\alpha}\leq r$,
			\begin{align}
				&\|G(u)\|_{\mathscr{L}_{2}(U,\mathcal{H}^{\alpha})}\leq l_{r}\|u\|_{\alpha},\notag \\
				&\|G'(u)\cdot v\|_{\mathscr{L}_{2}(U,\mathcal{H}^{\alpha})}\leq l_{r}\|v\|_{\alpha}, \label{eq26} \\
				&\|G''(u)\cdot (v,w)\|_{\mathscr{L}_{2}(U,\mathcal{H}^{\alpha})}\leq  l_{r}\|v\|_{\alpha}\|w\|_{\alpha},\notag\\
				& \|G'''(u)\cdot (v,w,x)\|_{\mathscr{L}_{2}(U,\mathcal{H}^{\alpha})}\leq  l_{r}\|v\|_{\alpha}\|w\|_{\alpha}\|x\|_{\alpha},\notag
			\end{align}
			where $l_{r}$ depends on $r$. Here $G'(u)$, $G''(u)$  and $G'''(u)$ are the first, second and third Fr$\acute{\textsl{e}}$chet derivatives with respect to $u$, respectively. Specially, $G'(0)(\cdot)$ is a linear operator: $\mathcal{H}^{\alpha}\rightarrow \mathscr{L}_{2}(U,\mathcal{H}^{\alpha})$.
		\end{Assumption}
		\begin{Definition}\label{def2}
			An $\mathcal{H}^{\alpha}$-valued process $u(t)$ is called  a local mild solution of problem $(\ref{eq10})$, if there exists some stopping time $\tau_{ex}$, we have  on a set of probability $1$ that $\tau_{ex}>0$,   $u\in\mathcal{C}^{0}([0,\tau_{ex}),\mathcal{H}^{\alpha}) $ and
			\begin{equation*}
				u(t)=e^{\mathcal{A}t}u(0)+\int^{t}_{0}e^{\mathcal{A}(t-s)}[\varepsilon^{2}\mathcal{L}u(t)+\mathcal{F}(u(t))]\textup{d}s+\varepsilon\int^{t}_{0}e^{\mathcal{A}(t-s)}G(u(t))\textup{d}W(s),
			\end{equation*}
			$\forall t\in[0,\tau_{ex})$.
		\end{Definition}
		
		In view of the cut-off technique and Theorem 7.2 in \cite{Da}, we can prove that (\ref{eq10}) has a unique local mild solution as the next theorem shown.
		\begin{Theorem}\label{theo3}
			Under Assumptions $\ref{assu1}$-$\ref{assu5}$, for any given $u(0)\in\mathcal{H}^{\alpha}$, there exists  a unique
			local mild solution of problem $(\ref{eq10})$ in the sense of Definition $\ref{def2}$ such that $\tau_{ex}=\infty$
			or $\lim\limits_{t\rightarrow \tau_{ex}}\|u(t)\|_{\alpha}=\infty$.
		\end{Theorem}
		
		We would like to state that a solution in the sense of Theorem
		\ref{theo3} with $\tau_{ex}$ maximal is always used in the paper.
		\section{Preliminary Lemmas}
		Introduce slow time-scale $T=\varepsilon^{2}t$ and decompose the local mild solution of $(\ref{eq10})$ as
		\begin{equation}\label{eq25}
			u(t)=\varepsilon a_{1}(T)+\varepsilon^{2} a_{2}(T)+\varepsilon \psi(T),
		\end{equation}
		where $a_{1}\in \mathcal{N}$ is the first-order term of slow modes, $a_{2}\in\mathcal{N}$ is the second-order term of slow modes, and $\psi\in\mathcal{S}$ is the fast modes.
		Since the main motivation of the paper is to capture the effective dynamics of $u(t)$ up to second-order approximation,
		we hope to pay more attention to the study of $a_{2}(T)$ and corresponding second-order amplitude equations. Hence, it is natural and intuitive to suppose that  $a_{1}(T)$ satisfy the first-order amplitude equations of $u(t)$ \cite{Fu1}:
		\begin{align}
			a_{1}(T)&= a_{1}(0)+ \int^{T}_{0}[\mathcal{L}_{c}a_{1}+\mathcal{F}_{c}(a_{1})]\textup{d}s
			+\int^{T}_{0}G'_{c}(0)(a_{1})\textup{d}\tilde{W},\label{aeq001}
		\end{align}
		where $\tilde{W}(T):=\varepsilon W(\varepsilon^{-2}T)$ is a rescaled version of the Wiener process $W(T)$ with the same law.
		Then,  the second-order term $a_{2}(T)$ satisfy
		\begin{align}
			a_{2}(T)&=a_{2}(0)+\int^{T}_{0}[\mathcal{L}_{c}a_{2}+\varepsilon^{-1}\mathcal{F}_{c}(a_{1}+\varepsilon a_{2}+\psi)-\varepsilon^{-1}\mathcal{F}_{c}(a_{1})]\textup{d}s\notag\\
			&~~+\int^{T}_{0}[\varepsilon^{-2}G_{c}(\varepsilon a_{1}+\varepsilon^{2} a_{2}+\varepsilon \psi)-\varepsilon^{-1}G'_{c}(0)(a_{1})]\textup{d}\tilde{W},\label{aeq003}
		\end{align}	
		and the fast modes satisfy
		\begin{align}
			\psi(T)=\mathcal{Q}(T)+J(T)+V(T),\label{aeq004}
		\end{align}
		where
		\begin{align*}
			\mathcal{Q}(T)&=e^{\mathcal{A}_{s}T\varepsilon^{-2}}\psi(0),\\
			J(T)&=\int^{T}_{0}e^{\mathcal{A}_{s}(T-s)\varepsilon^{-2}}[\mathcal{L}_{s}\psi
			+\mathcal{F}_{s}(a_{1}+\varepsilon a_{2}+\psi)]\textup{d}s,\\
			V(T)&=\varepsilon^{-1}\int^{T}_{0}e^{\mathcal{A}_{s}(T-s)\varepsilon^{-2}}[G_{s}(\varepsilon a_{1}+\varepsilon^{2} a_{2}+\varepsilon \psi)]\textup{d}\tilde{W}.
		\end{align*}
		\subsection{Estimations on Fast Modes}
		We are going to provide some estimations about $\psi(T)$, which can allow us to analyze $a_{2}(T)$ with more convenience. Let us introduce a stopping time  $\tau^{\star}$ such that $a_{1}(T)$, $a_{2}(T)$ and $\psi(T)$ do not blow up for  $T\in[0,\tau^{\star}]$.
		\begin{Definition}\label{def1}
			For an $\mathcal{N}\times\mathcal{S}$-valued stochastic process $(a_{1}+\varepsilon a_{2},\psi)$ given by $(\ref{aeq001})$, $(\ref{aeq003})$ and $(\ref{aeq004})$, we define, for $T_{0}>0$ and $\kappa\in(0,\frac{1}{20})$, the stopping time
			$\tau^{\star}$ as
			\begin{equation*}
				\tau^{\star}:=T_{0}\wedge\inf\{T>0,\|a_{1}(T)\|_{\alpha}>\varepsilon^{-\kappa}~or~ \|a_{2}(T)\|_{\alpha}>\varepsilon^{-\kappa}~or~\|\psi(T)\|_{\alpha}>\varepsilon^{-\kappa}
				\}.
			\end{equation*}
		\end{Definition}
		\begin{Lemma}\label{ale002}
			Under Assumptions $\ref{assu1}$-$\ref{assu5}$, for $p>1$, there exists a constant $C>0$, such that
			\begin{align*}
				\mathbb{E}\Big(\sup_{0\leq T\leq\tau^{\star}}\|J(T)\|_{\alpha}^{p}\Big)\leq C\varepsilon^{2p-3\kappa p}.
			\end{align*}
		\end{Lemma}
		\begin{proof}
			Letting $\hat{J}=\mathcal{L}_{s}\psi
			+\mathcal{F}_{s}(a_{1}+\varepsilon a_{2}+\psi)$, it is easy to own
			\begin{align*}
				&~~~~\mathbb{E}\Big(\sup_{0\leq T\leq\tau^{\star}}\|J(T)\|_{\alpha}^{p}\Big)\\
				&\leq \mathbb{E}\sup_{0\leq T\leq\tau^{\star}}\Big(\int^{T}_{0}\|e^{\mathcal{A}_{s}(T-s)\varepsilon^{-2}}\hat{J}\|_{\alpha}\textup{d}s\Big)^{p}\\
				&\leq C\varepsilon^{\frac{2\beta p}{m}}\mathbb{E}\sup_{0\leq T\leq \tau^{\star}}\Big(\int^{T}_{0}
				e^{-\varepsilon^{-2}\rho(T-s)}(T-s)^{-\frac{\beta}{m}}\|\hat{J}\|_{\alpha-\beta}
				\textup{d}s\Big)^{p}\\
				&\leq C\varepsilon^{\frac{2\beta p}{m}-3\kappa p}\sup_{0\leq T\leq \tau^{\star}}\Big(\int^{T}_{0}e^{-\varepsilon^{-2}\rho(T-s)}(T-s)^{-\frac{\beta}{m}}\textup{d}s\Big)^{p}\\
				&\leq C\varepsilon^{2p-3\kappa p},
			\end{align*}
			where the second inequality holds by Lemma \ref{l51}, and the third one holds due to Assumptions \ref{assu2} and \ref{assu3}.
		\end{proof}
		\begin{Lemma}\label{ale003}
			Under Assumptions $\ref{assu1}$-$\ref{assu5}$, for $p>1$, there exists a constant $C>0$, such that
			\begin{align*}
				\mathbb{E}\Big(\sup_{0\leq T\leq\tau^{\star}}\|V(T)\|_{\alpha}^{p}\Big)\leq C\varepsilon^{p-2\kappa p},
			\end{align*}
		\end{Lemma}
		\begin{proof}
			This proof is based on factorization method and fairly standard, but we provide it for the sake of completeness.
			For $p>2$, let $\gamma\in (\frac{1}{p},\frac{1}{2})$, and introduce
			\begin{align}
				\mathcal{D}(T)=\int^{T}_{0}(T-s)^{-\gamma}e^{\mathcal{A}_{s}(T-s)\varepsilon^{-2}}[G_{s}(\varepsilon a_{1}+\varepsilon^{2} a_{2}+\varepsilon \psi)]\textup{d}\tilde{W}.
			\end{align}
			By Stochastic Fubini Theorem, we obtain
			\begin{align*}
				\int^{T}_{0}e^{\mathcal{A}_{s}(T-s)\varepsilon^{-2}}[G_{s}(\varepsilon a_{1}+\varepsilon^{2} a_{2}+\varepsilon \psi)]\textup{d}\tilde{W}=C_{\gamma}\int^{T}_{0}(T-s)^{\gamma-1}e^{\mathcal{A}_{s}(T-s)\varepsilon^{-2}}\mathcal{D}(s)\textup{d}s,
			\end{align*}
			where $C_{\gamma}$ is a constant dependent of $\gamma$.
			Then, in view of Lemma \ref{l51} and H\"{o}lder inequality, we have
			\begin{align}
				\|V(T)\|_{\alpha}^{p}&\leq C\varepsilon^{-p}\Big(\int^{T}_{0}(T-s)^{\gamma-1}e^{-\rho(T-s)\varepsilon^{-2}}\|\mathcal{D}(s)\|_{\alpha}\textup{d}s\Big)^{p}\notag\\
				&\leq C\varepsilon^{-p}\Big(\int^{T}_{0}(T-s)^{\frac{p(\gamma-1)}{p-1}}e^{\frac{-\rho(T-s)\varepsilon^{-2}p}{p-1}}\textup{d}s\Big)^{p-1}\int^{T}_{0}\|\mathcal{D}(s)\|_{\alpha}^{p}\textup{d}s\notag\\
				&\leq C\varepsilon^{2\gamma p-2-p}\int^{T}_{0}\|\mathcal{D}(s)\|_{\alpha}^{p}\textup{d}s.\label{aeq005}
			\end{align}	
			Furthermore,  we obtain that for $T\in [0,\tau^{\star}]$,
			\begin{align}
				\mathbb{E}\|\mathcal{D}(T)\|_{\alpha}^{p}&\leq C\mathbb{E}\Big(\int^{T}_{0}(T-s)^{-2\gamma}\|e^{\mathcal{A}_{s}(t-s)\varepsilon^{-2}       }\bar{G}_{s}(\varepsilon a_{1}+\varepsilon^{2}a_{2}+\varepsilon\psi)\|^{2}_{\mathscr{L}_{2}(U,\mathcal{H}^{\alpha})  }\textup{d}s\Big)^{\frac{p}{2}}\notag\\
				&\leq C\varepsilon^{p}\mathbb{E}\Big(\int^{T}_{0}(T-s)^{-2\gamma}e^{-2\rho(t-s)\varepsilon^{-2}       }\mathbb{I}_{[0,\tau^{\star}]}(s)\|a_{1}+\varepsilon a_{2}+\psi\|^{2}_{\alpha}\textup{d}s\Big)^{\frac{p}{2}}\notag\\
				&\leq C\varepsilon^{2p-2\gamma p}\mathbb{E}\Big(\sup_{0\leq T\leq \tau^{\star}}\|a_{1}(T)+\varepsilon a_{2}(T)+\psi(T)\|_{\alpha}^{p}\Big)\notag\\
				&\leq C\varepsilon^{2p-2\gamma p-\kappa p},\label{aeq006}
			\end{align}
			where the first inequality holds because of Burkholder-Davis-Gundy inequality without the supremum
			in time (Lemma 7.2 of \cite{Da}), and the second inequality holds due to Lemma \ref{l51} and Assumption \ref{assu5}.
			
			By (\ref{aeq005}) and (\ref{aeq006}), we obtain that for $p>2$,
			\begin{align*}
				\mathbb{E}\Big(\sup_{0\leq T\leq \tau^{\star}}\|V(T)\|_{\alpha}^{p}\Big)\leq C\varepsilon^{p-\kappa p-2}.
			\end{align*}
			Obviously, this proof can be completed by H\"{o}lder inequality.
		\end{proof}
		
		According to Lemmas \ref{ale002} and \ref{ale003}, we  show the boundedness of $\psi$ as follows.
		\begin{Lemma}\label{ale004}
			Under Assumptions $\ref{assu1}$-$\ref{assu5}$,  for $a_{1}(T)$ defined in $(\ref{aeq001})$, and $p>1$, there exists a constant $C>0$, such that
			\begin{align}
				&\mathbb{E}\Big(\sup_{0\leq T \leq \tau^{\star}} \|\psi(T)\|_{\alpha}^{p}\Big)\leq C(\varepsilon^{p-3\kappa p}+\|\psi(0)\|_{\alpha}^{p}),\label{aeq007}\\
				&\mathbb{E}\Big(\sup_{0\leq T \leq \tau^{\star}} \|K(T)\|_{\alpha}^{p}\Big)\leq C(\varepsilon^{p-2\kappa p}+\|\psi(0)\|_{\alpha}^{p}\varepsilon^{p-kp}),\label{aeq007-1}\\
				&\mathbb{E}\Big(\sup_{0\leq T \leq \tau^{\star}} \|\psi(T)-\mathcal{Q}(T)-K(T)\|_{\alpha}^{p}\Big)\leq C\varepsilon^{2p-3\kappa p},\label{aeq008}
			\end{align}
			where 	$K(T)=\int^{T}_{0}e^{\mathcal{A}_{s}(T-s)\varepsilon^{-2}}G'_{s}(0)(a_{1}+\mathcal{Q})\textup{d}\tilde{W}$.
		\end{Lemma}
		\begin{proof}
			It is easy to obtain  (\ref{aeq007}) by Lemmas \ref{ale002} and \ref{ale003}. The proof of (\ref{aeq007-1}) is similar to that of Lemma (\ref{ale003}). And, we can also prove (\ref{aeq008}) with the help of Assumption \ref{assu5} and factorization method. In fact, we derive that
			\begin{align}
				&~~~~\mathbb{E}\Big(\sup_{0\leq T \leq \tau^{\star}} \|	V(T)-K(T)\|_{\alpha}\Big)^{p}\notag\\
				&\leq C\varepsilon^{p-2}\mathbb{E}\sup_{0\leq T \leq \tau^{\star}}\|\varepsilon^{-1}G(\varepsilon a_{1}+\varepsilon^{2} a_{2}+\varepsilon \psi)-G'(0)(a_{1}+\mathcal{Q})\|^{p}_{\mathscr{L}_{2}(U,\mathcal{H}^{\alpha})}.\label{aeq009}
			\end{align}
			By Taylor expansion, we notice that  there exists $\theta(T)$ on the line between $0$ and $\varepsilon(a_{1}(T)+ \varepsilon a_{2}(T)+\psi(T))$ such that
			\begin{align}
				&~~~~\varepsilon^{-1}G(\varepsilon a_{1}+\varepsilon^{2} a_{2}+\varepsilon\psi)-G'(0)(a_{1}+\mathcal{Q})\notag\\
				&=\varepsilon^{-1}[(G(0)+G'(0)(\varepsilon a_{1}+\varepsilon^{2} a_{2}+\varepsilon \psi)\notag\\
				&~~~~+\frac{1}{2}G''(\theta)(\varepsilon a_{1}+\varepsilon^{2} a_{2}+\varepsilon \psi,\varepsilon a_{1}+\varepsilon^{2} a_{2}+\varepsilon \psi)]-
				G'(0)(a_{1}+\mathcal{Q})\notag\\
				&= G'(0)(\varepsilon a_{2}+J+V)+\frac{\varepsilon}{2}G''(\theta)(a_{1}+\varepsilon a_{2}+\psi,a_{1}+\varepsilon a_{2}+\psi).\label{aeq010}
			\end{align}
			Then, combining (\ref{aeq009}) with (\ref{aeq010}), we complete the proof.
		\end{proof}
		
		Here, we would like to comment Lemma \ref{ale004}. From (\ref{aeq007}), we know the first-order approximation  of  $\psi(T)$ is just $\mathcal{Q}(T)$ (i.e., the information about the initial value of $\psi(T)$). As shown in (\ref{aeq008}),  we can improve the approximation  of $\psi$ up to second-order if we take account into  the information about the first-order amplitude equations $a_{1}(T)$.
		\subsection{Estimations on  Slow modes}
		This subsection is devoted to getting rid of the high-order terms from $a_{2}(T)$, and deriving the second-order amplitude equations. Let us  start with the following lemmas.
		\begin{Lemma}\label{ale005}
			Under Assumptions $\ref{assu1}$-$\ref{assu5}$, for $p>1$ $i,j\in\{1,2\}$, there exists a constant $C>0$ such that
			\begin{align*}
				&\mathbb{E}\Big(\sup_{0\leq T\leq\tau^{\star}}\|\int^{T}_{0}\mathcal{F}_{c}(a_{i},a_{j},\psi)\textup{d}s\|^{p}_{\alpha}\Big)\leq C(\varepsilon^{p-5\kappa p}+\|\psi(0)\|^{p}_{\alpha}\varepsilon^{2p-2\kappa p}),\\
				&\mathbb{E}\Big(\sup_{0\leq T\leq\tau^{\star}}\|\int^{T}_{0}\mathcal{F}_{c}(a_{i},\psi,\psi)\textup{d}s\|^{p}_{\alpha}\Big)\leq C(\varepsilon^{2p-7\kappa p}+\|\psi(0)\|^{2p}_{\alpha}\varepsilon^{2p-\kappa p}),\\
				&\mathbb{E}\Big(\sup_{0\leq T\leq\tau^{\star}}\|\int^{T}_{0}\mathcal{F}_{c}(\psi)\textup{d}s\|^{p}_{\alpha}\Big)\leq C(\varepsilon^{3p-9\kappa p}+\|\psi(0)\|^{3p}_{\alpha}\varepsilon^{2p}).
			\end{align*}
		\end{Lemma}
		\begin{proof}
			Since the proof of the above inequalities are similar, we just prove the first one for simplicity.
			Recalling the fact that for $\alpha\in \mathbb{R}$, all $\mathcal{H}^{\alpha}$ norms are equivalent on $\mathcal{N}$, we obtain that
			\begin{align*}
				&~~~~\mathbb{E}\Big(\sup_{0\leq T\leq\tau^{\star}}\|\int^{T}_{0}\mathcal{F}_{c}(a_{i},a_{j},\psi)\textup{d}s\|^{p}_{\alpha}\Big)\\
				&\leq C\mathbb{E}\Big(\sup_{0\leq T\leq\tau^{\star}}\Big[\int^{T}_{0}\|\mathcal{F}_{c}(a_{i},a_{j},\psi)\|_{\alpha-\beta}\textup{d}s\Big]^{p}\Big)\\
				&\leq C\mathbb{E}\Big(\sup_{0\leq T\leq\tau^{\star}}\Big[\int^{T}_{0}\|a_{i}\|_{\alpha}\|a_{j}\|_{\alpha}(\|\mathcal{Q}\|_{\alpha}+\|J\|_{\alpha}+\|V\|_{\alpha})\textup{d}s\Big]^{p}\Big)\\
				&\leq C\varepsilon^{\frac{2\beta p}{m}-2\kappa p}\Big(\int^{T_{0}}_{0}s^{-\frac{\beta}{m}}e^{-\rho\varepsilon^{-2}s}\|\psi(0)\|_{\alpha-\beta}\textup{d}s\Big)^{p}+C\varepsilon^{p-5\kappa p}\\
				&\leq C(\varepsilon^{p-5\kappa p}+\|\psi(0)\|^{p}_{\alpha}\varepsilon^{2p-2\kappa p}),
			\end{align*}
			where  Assumption \ref{assu3} is used in the second inequality, Lemmas \ref{l51}, \ref{ale002} and \ref{ale003} are used in the third one.
		\end{proof}
		\begin{Lemma}\label{ale006}
			Under Assumptions $\ref{assu1}$-$\ref{assu5}$, for $p>1$, there exists a constant $C>0$ such that
			\begin{align*}
				\mathbb{E}\Big(\sup_{0\leq T\leq\tau^{\star}}\|\int^{T}_{0}\mathcal{F}_{c}(a_{1},a_{1},K)\textup{d}s\|^{p}_{\alpha}\Big)\leq C(\varepsilon^{2p-6\kappa p}+\|\psi(0)\|^{p}_{\alpha}\varepsilon^{2p-5\kappa p}).
			\end{align*}
		\end{Lemma}
		\begin{proof}
			For completing the proof, it is sufficient to prove that for $i=1,\cdots N$,
			\begin{align*}
				\mathbb{E}\Big(\sup_{0\leq T\leq\tau^{\star}}|	\int^{T}_{0}\langle\mathcal{F}_{c}(a_{1},a_{1},K),e_{i}\rangle\textup{d}s|^{p}_{\alpha}\Big)\leq C(\varepsilon^{2p-6\kappa p}+\|\psi(0)\|^{p}_{\alpha}\varepsilon^{2p-5\kappa p}),
			\end{align*}
			Let us achieve this goal in the followings.
			Recall $K(T)$ satisfies
			\begin{align*}
				\textup{d}K=\varepsilon^{-2}\mathcal{A}_{s}K\textup{d}T+G_{s}'(0)(\mathcal{Q}+a_{1})\textup{d}\tilde{W}.
			\end{align*}
			Applying It$\hat{\textup{o}}^{,}$s formula to $\langle\mathcal{F}_{c}(a_{1},a_{1},\mathcal{A}_{s}^{-1}K),e_{i}\rangle$, we derive that
			\begin{align*}
				&~~~~\langle\mathcal{F}_{c}(a_{1}(T),a_{1}(T),\mathcal{A}_{s}^{-1}K(T)),e_{i}\rangle-\langle\mathcal{F}_{c}(a_{1}(0),a_{1}(0),\mathcal{A}_{s}^{-1}K(0)),e_{i}\rangle\\
				&=\int^{T}_{0}[2\langle\mathcal{F}_{c}(\mathcal{L}_{c}a_{1}+\mathcal{F}_{c}(a_{1}),a_{1},\mathcal{A}_{s}^{-1}K),e_{i}\rangle+\varepsilon^{-2}\langle\mathcal{F}_{c}(a_{1},a_{1},K),e_{i}\rangle]\textup{d}s\\
				&~~~+\int^{T}_{0}\sum_{j=1}^{\infty}\langle\mathcal{F}_{c}(G_{c}'(0)(a_{1})f_{j},G_{c}'(0)(a_{1})f_{j},\mathcal{A}_{s}^{-1}K),e_{i}\rangle\textup{d}s\\
				&~~~+2\int^{T}_{0}\sum_{j=1}^{\infty}\langle\mathcal{F}_{c}(G_{c}'(0)(a_{1})f_{j},a_{1},\mathcal{A}_{s}^{-1}G_{s}'(0)(\mathcal{Q}+a_{1})f_{j}),e_{i}\rangle\textup{d}s\\
				&~~~+2\int^{T}_{0}\langle\mathcal{F}_{c}(a_{1},G_{c}'(0)(a_{1})\cdot,\mathcal{A}_{s}^{-1}K),e_{i}\rangle\textup{d}\tilde{W}\\
				&~~~+\int^{T}_{0}\langle\mathcal{F}_{c}(a_{1},a_{1},\mathcal{A}_{s}^{-1}G_{s}'(0)(\mathcal{Q}+a_{1})\cdot),e_{i}\rangle\textup{d}\tilde{W}
			\end{align*}
			Noting $\mathcal{A}_{s}^{-1}$ is a bounded linear operator from $\mathcal{H}^{\alpha}$ to  $\tilde{\mathcal{H}}$ with $\tilde{\mathcal{H}}\subset\mathcal{H}^{\alpha}$,
			we obtain that
			\begin{align*}
				&~~~~\mathbb{E}\Big(\sup_{0\leq T\leq T^{\star}}|\int^{T}_{0}\langle\mathcal{F}_{c}(\mathcal{L}_{c}a_{1}+\mathcal{F}_{c}(a_{1}),a_{1},\mathcal{A}_{s}^{-1}K),e_{i}\rangle\textup{d}s|^{p}\Big)\\
				&\leq C\mathbb{E}\Big(\sup_{0\leq T\leq T^{\star}}[\int^{T}_{0}\|a_{1}\|_{\alpha}^{4}\|K\|_{\alpha}\textup{d}s
				]^{p}\Big)\\
				&\leq C(\varepsilon^{p-6\kappa p}+\|\psi(0)\|_{\alpha}^{p}\varepsilon^{p-5\kappa p}),
			\end{align*}
			where we use (\ref{aeq007-1}) in the last inequality. Similarly, other drift terms can  be estimated.
			Using  Burkholder-Davis-Gundy inequality, we can also  estimate diffusion terms as follows:
			\begin{align*}
				&~~~~\mathbb{E}\Big(\sup_{0\leq T\leq T^{\star}}|\int^{T}_{0}\langle\mathcal{F}_{c}(a_{1},G_{c}'(0)(a_{1})\cdot,\mathcal{A}_{s}^{-1}K),e_{i}\rangle\textup{d}\tilde{W}|^{p}\Big)\\
				&\leq  C(\varepsilon^{p-4\kappa p}+\|\psi(0)\|_{\alpha}^{p}\varepsilon^{p-3\kappa p}),\\
				&~~~~\mathbb{E}\Big(\sup_{0\leq T\leq T^{\star}}|\int^{T}_{0}\langle\mathcal{F}_{c}(a_{1},a_{1},\mathcal{A}_{s}^{-1}G_{s}'(0)(\mathcal{Q}+a_{1})\cdot),e_{i}\rangle\textup{d}\tilde{W}|^{p}\Big)\\
				&\leq  C(\varepsilon^{-3\kappa p}+\|\psi(0)\|_{\alpha}^{p}\varepsilon^{p-2\kappa p}).
			\end{align*}
			Collecting all the estimations, we complete the proof.
		\end{proof}
		
		By Lemma  \ref{ale006}, it is straightforward to have the next result.
		\begin{Lemma}\label{le301}
			Under Assumptions $\ref{assu1}$-$\ref{assu5}$, for $p>1$, there exists a constant $C>0$ such that
			\begin{align*}
				\mathbb{E}\Big(\sup_{0\leq T\leq\tau^{\star}}\|\int^{T}_{0}\mathcal{F}_{c}(a_{1},a_{1},\psi)\textup{d}s\|^{p}_{\alpha}\Big)\leq C(\varepsilon^{2p-6\kappa p}+\|\psi(0)\|^{p}_{\alpha}\varepsilon^{2p-5\kappa p}).
			\end{align*}
		\end{Lemma}
		\begin{Lemma}\label{ale009}
			Under Assumptions $\ref{assu1}$, $\ref{assu4}$, $\ref{assu5}$, for $T_{0}>0$ and $p>1$, there exists a constant $C>0$ such that
			\begin{align*}
				\mathbb{E}\Big(\sup_{0\leq T
					\leq T_{0}}\|\int^{T}_{0}G_{c}'(0)[\int^{s}_{0}e^{\mathcal{A}_{s}(s-r)\varepsilon^{-2}}G_{s}'(0)(\mathcal{Q})\textup{d}\tilde{W}(r)]\textup{d}\tilde{W}(s)\|_{\alpha}^{p}\Big)\leq C\|\psi(0)\|_{\alpha}^{p} \varepsilon^{2p}.
			\end{align*}
		\end{Lemma}
		\begin{proof}
			For $p>2$,  we obtain that
			\begin{align*}
				&~~~~\mathbb{E}\Big(\sup_{0\leq T\leq T_{0}}\|\int^{T}_{0}G_{c}'(0)[\int^{s}_{0}e^{\mathcal{A}_{s}(s-r)\varepsilon^{-2}}G_{s}'(0)(\mathcal{Q})\textup{d}\tilde{W}(r)]\textup{d}\tilde{W}(s)\|_{\alpha}^{p}\Big)\\
				&\leq C\Big(\int^{T_{0}}_{0}\big(\mathbb{E}\|G_{c}'(0)[\int^{s}_{0}e^{\mathcal{A}_{s}(s-r)\varepsilon^{-2}}G_{s}'(0)(\mathcal{Q})\textup{d}\tilde{W}(r)]\|_{\mathscr{L}_{2}(U,\mathcal{H}^{\alpha}) }^{p}\big)^\frac{2}{p}\textup{d}s\Big)^{\frac{p}{2}}\\
				&\leq C\Big(\int^{T_{0}}_{0}\big(\mathbb{E}\|\int^{s}_{0}e^{\mathcal{A}_{s}(s-r)\varepsilon^{-2}}G_{s}'(0)(\mathcal{Q})\textup{d}\tilde{W}(r)\|_{\alpha }^{p}\big)^\frac{2}{p}\textup{d}s\Big)^{\frac{p}{2}}\\
				&\leq C\mathbb{E}\Big(\int^{T_{0}}_{0}\int^{s}_{0} \|e^{\mathcal{A}_{s}(s-r)\varepsilon^{-2}}G_{s}'(0)(\mathcal{Q})\|_{\mathscr{L}_{2}(U,\mathcal{H}^{\alpha})}^{2}\textup{d}r\textup{d}s\Big)^{\frac{p}{2}}\\
				&=C\mathbb{E}\Big(\int^{T_{0}}_{0}\int^{s}_{0} \sum^{\infty}_{j=1}\|e^{\mathcal{A}_{s}(s-r)\varepsilon^{-2}}G_{s}'(0)(\mathcal{Q})f_{j}\|_{\alpha}^{2}\textup{d}r\textup{d}s\Big)^{\frac{p}{2}}\\
				&\leq C\mathbb{E}\Big(\int^{T_{0}}_{0}\int^{s}_{0} e^{-2\rho(s-r)\varepsilon^{-2}}\sum^{\infty}_{j=1}\|G_{s}'(0)(\mathcal{Q})f_{j}\|_{\alpha}^{2}\textup{d}r\textup{d}s\Big)^{\frac{p}{2}}\\
				&=C\mathbb{E}\Big(\int^{T_{0}}_{0}\int^{s}_{0}e^{-\rho(s-r)\varepsilon^{-2}} \|G_{s}'(0)(\mathcal{Q})\|_{\mathscr{L}_{2}(U,\mathcal{H}^{\alpha})}^{2}\textup{d}r\textup{d}s\Big)^{\frac{p}{2}}\\
				& \leq C\mathbb{E}\Big(\int^{T_{0}}_{0}\int^{s}_{0}e^{-\rho(s-r)\varepsilon^{-2}} \|\mathcal{Q}\|_{\alpha}\textup{d}r\textup{d}s\Big)^{\frac{p}{2}}\\
				&  \leq C\mathbb{E}\Big(\int^{T_{0}}_{0}e^{-\rho s\varepsilon^{-2}}s \textup{d}s\Big)^{\frac{p}{2}}\|\psi(0)\|_{\alpha}^{p}\\
				&\leq C\|\psi(0)\|_{\alpha}^{p} \varepsilon^{2p},
			\end{align*}
			where we use Lemma 7.7 of \cite{Da} in the second line, and Burkholder-Davis-Gundy inequality in the forth one.   The case for $1<p\leq2 $ still holds due to H\"{o}lder inequality.
		\end{proof}
		\begin{Lemma}\label{ale008}
			Under Assumptions $\ref{assu1}$-$\ref{assu5}$, for $p>1$ and $\|\psi(0)\|_{\alpha}\leq \varepsilon$, there exists a constant $C>0$ such that
			\begin{align*}
				\mathbb{E}\Big(\sup_{0\leq T\leq\tau^{\star}}\|\int^{T}_{0} \mathcal{M}(s)  \textup{d}\tilde{W}\|_{\alpha}^{p}\Big)\leq C\varepsilon^{p-9\kappa p},
			\end{align*}
			where
			\begin{align*}
				\mathcal{M}(T)&=\varepsilon^{-2}G_{c}(\varepsilon a_{1}(T)+\varepsilon^{2} a_{2}(T)+\varepsilon \psi(T))-\varepsilon^{-1}G'_{c}(0)(a_{1}(T))-G'_{c}(0)(a_{2}(T))\\
				&~~~
				-\varepsilon^{-1}G'_{c}(0)(\int^{T}_{0}e^{\mathcal{A}_{s}(T-s)\varepsilon^{-2}}G_{s}'(0)(a_{1})\textup{d}\tilde{W})-\frac{1}{2}G''_{c}(0)(a_{1}(T),a_{1}(T))\\
			\end{align*}
		\end{Lemma}
		\begin{proof}
			By Taylor expansion, we know that there exists $\tilde{\theta}(t)$  on the line between $0$ and $\varepsilon(a_{1}(T)+ \varepsilon a_{2}(T)+\psi(T))$ such that
			\begin{align}
				&~~~~\varepsilon^{-2}G(\varepsilon a_{1}+\varepsilon^{2} a_{2}+\varepsilon\psi)\notag\\
				&=\varepsilon^{-1}G'(0)(a_{1}+\psi)+
				G'(0)(a_{2})+\frac{1}{2}G''(0)(a_{1}+\varepsilon a_{2}+\psi,a_{1}+\varepsilon a_{2}+\psi)\notag
				\\
				&~~~+\frac{\varepsilon}{6}G'''(\tilde{\theta})(a_{1}+\varepsilon a_{2}+\psi,a_{1}+\varepsilon a_{2}+\psi,a_{1}+\varepsilon a_{2}+\psi).\label{aeq011}
			\end{align}
			Based on (\ref{aeq011}), we turn to prove that
			\begin{align}\label{aeq012}
				\mathbb{E}\Big(\sup_{0\leq T\leq\tau^{\star}}\|\int^{T}_{0} \bar{\mathcal{M}}(s)  \textup{d}\tilde{W}\|_{\alpha}^{p}\Big)\leq C\varepsilon^{p-9\kappa p},
			\end{align}
			where
			\begin{align*}
				\bar{\mathcal{M}}(T)&=\varepsilon^{-1}G_{c}'(0)[\mathcal{Q}(T)+J(T)+V(T)-K(T)]\\
				&~~~+\varepsilon^{-1}G_{c}'(0)(\int^{T}_{0}e^{\mathcal{A}_{s}(s-r)\varepsilon^{-2}}G_{s}'(0)(\mathcal{Q})\textup{d}\tilde{W}(s))\\
				&~~~+\frac{1}{2}G''_{c}(0)(a_{1}+\varepsilon a_{2}+\psi,a_{1}+\varepsilon a_{2}+\psi)-\frac{1}{2}G''(0)(a_{1},a_{1})\\
				&~~~+\frac{\varepsilon}{6}G'''_{c}(\tilde{\theta})(a_{1}+\varepsilon a_{2}+\psi,a_{1}+\varepsilon a_{2}+\psi,a_{1}+\varepsilon a_{2}+\psi).
			\end{align*}
			By Burkholder-Davis-Gundy inequality, we can obtain
			\begin{align*}
				&~~~~\mathbb{E}\Big(\sup_{0\leq T\leq\tau^{\star}}\|\int^{T}_{0}  \varepsilon^{-1}G'_{c}(0)(\mathcal{Q})\textup{d}\tilde{W}\|_{\alpha}^{p}\Big)\\
				&\leq C\|\psi(0)\|_{\alpha}^{p}\\
				&\leq C\varepsilon^{p},
			\end{align*}
			and
			\begin{align*}
				&~~~~\mathbb{E}\Big(\sup_{0\leq T\leq\tau^{\star}}\|\int^{T}_{0}  G''_{c}(0)(\psi,\psi)\textup{d}\tilde{W}\|_{\alpha}^{p}\Big)\\
				&\leq C \mathbb{E}\Big(\int^{T_{0}}_{0} \mathbb{I}_{[0,\tau^{\star}]}  \|G''_{c}(0)(\psi,\psi)\|_{\mathscr{L}_{2}(U,\mathcal{H}^{\alpha}) }^{2}\textup{d}s            \Big)^{\frac{p}{2}}\\
				&\leq C \mathbb{E}\Big(\int^{T_{0}}_{0} \mathbb{I}_{[0,\tau^{\star}]}  \|\psi\|_{\alpha}^{2}\textup{d}s            \Big)^{\frac{p}{2}}\\
				&\leq C \mathbb{E}\Big(\int^{T_{0}}_{0} \mathbb{I}_{[0,\tau^{\star}]}\big( \|\mathcal{Q}\|_{\alpha}^{2}+\|J\|_{\alpha}^{2}+\|V\|_{\alpha}^{2}\big)\textup{d}s            \Big)^{\frac{p}{2}}\\
				&\leq C\varepsilon^{2p-3\kappa p}.
			\end{align*}
			The estimations of other  terms can be obtained in view of  Lemmas \ref{ale002}, \ref{ale003} and \ref{ale009}.
		\end{proof}
		\begin{Remark}  In \textup{\cite{Fu1}}, the initial value of fast modes is assumed to be of order $\varepsilon^{-\frac{\kappa}{3}}$, but we need $\psi(0)$ is of order $\varepsilon$ as provided in Lemma \textup{\ref{ale008}}.  Indeed, it is  reasonable and plausible to adopt this assumption. If else,  fast modes could be involved in the second-order amplitude equations. However, our aim is to obtain the approximation of SPDEs via SDEs, so it is better to study the amplitude equations  without information about the fast modes.  The other reason is from the technical aspect. In other words,  $\int^{T}_{0}\varepsilon^{-1}G'_{c}(0)(\mathcal{Q})\textup{d}\tilde{W}$ is of order $\mathcal{O}(\|\psi(0)\|_{\alpha})$, and we can not  separate any high-order terms from it.
			Nevertheless, an alternative assumption is $G'_{c}(0)(\psi(0))\mathcal{U}=0$, for any $\mathcal{U}\in U$,
			but we  dismiss this case  for the sake of simplicity.
		\end{Remark}
		
		Thanks to previous preliminaries, we now can separate some high-order terms from $a_{2}(T)$.
		\begin{Lemma}\label{ale010}
			Under Assumptions $\ref{assu1}$-$\ref{assu5}$, for $p>1$ and $\|\psi(0)\|_{\alpha}\leq \varepsilon$, there exists a constant $C>0$ such that
			\begin{align}
				a_{2}(T)&=a_{2}(0)+\int^{T}_{0}[\mathcal{L}_{c}a_{2}+3\mathcal{F}_{c}(a_{1},a_{1},a_{2})]\textup{d}s+\int^{T}_{0}\varepsilon^{-1}G'_{c}(0)(\mathcal{Y})\textup{d}\tilde{W}\notag\\
				&~~~+\int^{T}_{0}[G_{c}'(0)(a_{2})+\frac{1}{2}G_{c}''(0)(a_{1},a_{1})]\textup{d}\tilde{W}
				+R_{1}(T),\label{aeq013}
			\end{align}
			where
			\begin{align*}
				\mathcal{Y}(T)=	\int^{T}_{0}e^{\mathcal{A}_{s}(T-s)\varepsilon^{-2}}G_{s}'(0)(a_{1})\textup{d}\tilde{W},
			\end{align*}
			and
			\begin{align*}
				\mathbb{E}\Big(\sup_{0\leq T\leq\tau^{\star}}\|R_{1}(T)\|_{\alpha}^{p}\Big)\leq C\varepsilon^{p-9\kappa p}.
			\end{align*}
		\end{Lemma}
		\begin{proof}
			By Lemmas \ref{ale005} and \ref{le301}, we can obtain that
			\begin{align*}
				\int^{T}_{0}[\varepsilon^{-1}\mathcal{F}_{c}(a_{1}+\varepsilon a_{2}+\psi)-\varepsilon^{-1}\mathcal{F}_{c}(a_{1})-3\mathcal{F}_{c}(a_{1},a_{1},a_{2})]\textup{d}s=\mathcal{O}(\varepsilon^{1-9\kappa}).
			\end{align*}
			Then, using Lemma \ref{ale008} to bound the high-order terms of  the diffusion term, we complete the proof.
		\end{proof}

		With the help of Lemma \ref{ale010}, we  get rid of the high-order terms $R_{1}(T)$ from (\ref{aeq013}), and derive the reduced system of $a_{2}(T)$ as follows:
		\begin{align}
			b(T)&=a_{2}(0)+\int^{T}_{0}[\mathcal{L}_{c}b+3\mathcal{F}_{c}(a_{1},a_{1},b)]\textup{d}s+\int^{T}_{0}\varepsilon^{-1}G'_{c}(0)(\mathcal{Y})\textup{d}\tilde{W}\notag\\
			&~~~+\int^{T}_{0}[G_{c}'(0)(b)+\frac{1}{2}G_{c}''(0)(a_{1},a_{1})]\textup{d}\tilde{W}. \label{aeq015}
		\end{align}
		Note that there still exists the fast fluctuation $\int^{T}_{0}\varepsilon^{-1}G'_{c}(0)(\mathcal{Y})\textup{d}\tilde{W}$ in (\ref{aeq015}), which results in  that (\ref{aeq015}) can be not the second-order amplitude equations as desired.
		Before further analyzing such fast fluctuation term, let us present some estimations about $a_{1}(T)$, $a_{2}(T)$ and $b(T)$ as follows.
		\begin{Lemma}\label{ale001}
			Under Assumptions $\ref{assu1}$-$\ref{assu5}$, for $a_{1}(T)$ given in $(\ref{aeq001})$, $T_{0}>0$ and $p>1$, there exists a constant $C>0$, such that
			\begin{align*}
				\mathbb{E}\Big(\sup_{0\leq T \leq T_{0}} \|a_{1}(T)\|^{p}\Big)\leq C\|a_{1}(0)\|^{p}.
			\end{align*}
		\end{Lemma}
		\begin{proof}
			The proof follows from Lemma 4.8 in \cite{Fu1}.
		\end{proof}
		\begin{Lemma}\label{ale016}
			Under Assumptions $\ref{assu1}$-$\ref{assu5}$, for $a_{1}(T)$ given in $(\ref{aeq001})$, $T_{0}>0$ and $p>1$, there exists a constant $C>0$, such that
			\begin{align*}
				\mathbb{E}\Big(\sup_{0\leq T \leq T_{0}} \int^{T}_{0}\|G_{c}'(0)\int^{s}_{0}[e^{\mathcal{A}_{s}(s-r)\varepsilon^{-2}}G_{s}'(0)(a_{1})]\textup{d}\tilde{W}\|_{\mathscr{L}_{2}(U,\mathcal{H}^{\alpha})}^{p}\textup{d}s\Big)\leq C\varepsilon^{p}\|a_{1}(0)\|^{p}.
			\end{align*}
		\end{Lemma}
		\begin{proof}
			For $p>1$, we obtain that
			\begin{align*}
				&~~~~\mathbb{E}\Big(\sup_{0\leq T \leq T_{0}} \int^{T}_{0}\|G_{c}'(0)\int^{s}_{0}[e^{\mathcal{A}_{s}(s-r)\varepsilon^{-2}}G_{s}'(0)(a_{1})]\textup{d}\tilde{W}\|^{p}_{\mathscr{L}_{2}(U,\mathcal{H}^{\alpha})}\textup{d}s\Big)\\
				&\leq  \int^{T_{0}}_{0}\mathbb{E}\|G_{c}'(0)\int^{s}_{0}[e^{\mathcal{A}_{s}(s-r)\varepsilon^{-2}}G_{s}'(0)(a_{1})]\textup{d}\tilde{W}\|_{\mathscr{L}_{2}(U,\mathcal{H}^{\alpha})}^{p}\textup{d}s\\
				&\leq C  \int^{T_{0}}_{0}\mathbb{E}\|\int^{s}_{0}[e^{\mathcal{A}_{s}(s-r)\varepsilon^{-2}}G_{s}'(0)(a_{1})]\textup{d}\tilde{W}(r)\|_{\alpha}^{p}\textup{d}s\\
				&\leq C \int^{T_{0}}_{0}\mathbb{E}[\int^{s}_{0}(e^{2\rho(s-r)\varepsilon^{-2}}\|a_{1}\|^{2})\textup{d}r]^{\frac{p}{2}}\textup{d}s\\
				&\leq C\varepsilon^{p}\|a_{1}(0)\|^{p},
			\end{align*}
			where the second inequality holds from Assumption \ref{assu5}, the third one holds due to Burkholder-Davis-Gundy inequality, and the last one holds by Lemma \ref{ale001}.
		\end{proof}
		
		\begin{Lemma}\label{ale021}
			Under Assumptions $\ref{assu1}$-$\ref{assu5}$,   for $a_{2}(T)$ and $b(T)$ given in $(\ref{aeq013})$ and $(\ref{aeq015})$ respectively, $T_{0}>0$, $p>1$ and $\|\psi(0)\|_{\alpha}\leq \varepsilon$, there exists a constant $C>0$, such that
			\begin{align}
				&\mathbb{E}\Big(\sup_{0\leq T \leq T_{0}} \|b(T)\|^{p}\Big)\leq C(1+\|a_{1}(0)\|^{2p}+\|a_{2}(0)\|^{p}
				),\label{aeq017}\\
				&\mathbb{E}\Big(\sup_{0\leq T \leq T_{0}} \|b(T)-a_{2}(T)\|^{p}\Big)\leq C \varepsilon^{p-18\kappa p}(1+\|a_{1}(0)\|^{2p})\label{aeq019},\\
				&\mathbb{E}\Big(\sup_{0\leq T \leq T_{0}} \|a_{2}(T)\|^{p}\Big)\leq C (1+\|a_{1}(0)\|^{2p}+\|a_{2}(0)\|^{p}
				).  \label{aeq020}
			\end{align}
		\end{Lemma}
		\begin{proof}
			This proof of (\ref{aeq017}) is similar to Lemma \ref{ale001}, but we would like to show it for the completeness.
			
			Set the stopping time $\tau_{\mathcal{K}}$:
			\begin{equation*}
				\tau_{\mathcal{K}}:=\inf\{T>0, \|b(T)\|>\mathcal{K}\}.
			\end{equation*}
			For $p\geq 2$ and $T\in[0,T_{0}\wedge \tau_{\mathcal{K}}]$, applying It$\hat{\textup{o}}^{,}$s formula to (\ref{aeq015}), we obtain
			\begin{align*}
				\|b(T)\|^{p}&\leq C\|a_{2}(0)\|^p+ C\int^{T}_{0}\|b\|^{p-2}[\langle\mathcal{L}_{c}b,b\rangle+\langle\mathcal{F}_{c}(a_{1},a_{1},b),b\rangle]\textup{d}s\\
				&~~~+C\int^{T}_{0}\|b\|^{p-2}\langle b,G_{c}'(0)(b)+G_{c}''(0)(a_{1},a_{1})+\varepsilon^{-1}G'_{c}(0)(\mathcal{Y})\textup{d}\tilde{W}\rangle\\
				&~~~+C\int^{T}_{0}\|b\|^{p-2}\|G_{c}'(0)(b)+G_{c}''(0)(a_{1},a_{1})+\varepsilon^{-1}G'_{c}(0)(\mathcal{Y})\|^{2}_{\mathscr{L}_{2}(U,\mathcal{H}^{\alpha})}\textup{d}s.
			\end{align*}
			Then, we obtain that for any $T_{1}\in[0,T_{0}]$:
			\begin{align*}
				&~~~~\mathbb{E}\Big(\sup_{0\leq T\leq T_{1}\wedge\tau_{\mathcal{K}}}\|b(T)\|^{p}\Big)\\
				&\leq C\|a_{2}(0)\|^p+C\mathbb{E}\Big(\sup_{0\leq T\leq T_{0}}\|a_{1}(t)\|^{2p}\Big)+C\mathbb{E}\Big(\int^{T_{1}\wedge\tau_{\mathcal{K}}}_{0}\|b_{1}(s)\|^{p}\textup{d}s\Big)\\
				&~~~+C\mathbb{E}\Big(\int^{T_{1}\wedge\tau_{\mathcal{K}}}_{0}\|b_{1}(s)\|^{2p}\textup{d}s\Big)^{\frac{1}{2}}+
				C\varepsilon^{p}\mathbb{E}\Big(\int^{T_{1}\wedge\tau_{\mathcal{K}}}_{0}\|G_{c}'(0)(
				\mathcal{Y}(s))\|_{\mathscr{L}_{2}(U,\mathcal{H}^{\alpha})}^{p}\textup{d}s\Big)
				\\
				&~~~+C\varepsilon^{-p}\mathbb{E}\Big(\int^{T_{1}\wedge\tau_{\mathcal{K}}}_{0}\|G_{c}'(0)(
				\mathcal{Y}(s))\|_{\mathscr{L}_{2}(U,\mathcal{H}^{\alpha})}^{2p}\textup{d}s\Big)^{\frac{1}{2}}\\
				&\leq C(1+\|a_{2}(0)\|^p+\|a(0)\|^{2p})
				+C\int^{T_{1}}_{0}\mathbb{E}\Big(\sup_{0\leq s_{1}\leq s\wedge\tau_{\mathcal{K}}}\|b(s_{1})\|^{p}\Big)\textup{d}s\\
				&~~~+\frac{1}{2}\mathbb{E}\Big(\sup_{0\leq T\leq T_{1}\wedge\tau_{\mathcal{K}}}\|b(T)\|^{p}\Big),
			\end{align*}
			where the first inequality follows from Cauchy-Schwarz inequality and Burkholder-Davis-Gundy inequality,
			the second one follows from Young's inequality and Lemma \ref{ale016}.
			Furthermore, using Gronwall's lemma, we derive that
			\begin{equation*}
				\mathbb{E}\Big(\sup_{0\leq T\leq T_{0}\wedge\tau_{\mathcal{K}}}\|b(T)\|^{p}\Big)\leq C(1+\|a_{2}(0)\|^p+\|a_{1}(0)\|^{2p})
				.
			\end{equation*}
			Next task is to replace $T_{0}\wedge\tau_{\mathcal{K}}$ with $T_{0}$.
			Denoting the right side of the last inequality by $\tilde{C}$, we obtain that
			\begin{align*}
				\quad\mathbb{P}\Big(\tau_{\mathcal{K}}>T_{0}\Big)&=\mathbb{P}\Big(\sup_{T\leq T_{0} \wedge\tau_{\mathcal{K}}}\|b(T)\|<\mathcal{K}\Big)\\
				&=1-\mathbb{P}\Big(\sup_{T\leq T_{0}\wedge\tau_{\mathcal{K}}}\|b(T)\|\geq \mathcal{K}\Big)\\
				&\geq 1-\frac{\tilde{C}}{\mathcal{K}^{p}},
			\end{align*}
			It follows from the above inequality that  if $\mathcal{K}\rightarrow \infty$,
			$\sup_{0\leq T\leq T_{0}\wedge\tau_{\mathcal{K}}}\|b(t)\|^{p}$ monotonously converges to $\sup_{0\leq T\leq T_{0}}\|b(t)\|^{p}~a.s..$
			Then, monotone convergence theorem yields that
			\begin{equation*}
				\mathbb{E}\Big(\sup_{0\leq T \leq T_{0}} \|b(T)\|^{p}\Big)=\lim_{\mathcal{K}\rightarrow \infty} \mathbb{E}\Big(\sup_{0\leq T \leq T_{0}\wedge\tau_{\mathcal{K}}}\|b(T)\|^{p}\Big)\leq \tilde{C}.
			\end{equation*}
			We can further use H$\textup{\"{o}}$lder inequality to prove  (\ref{aeq017}) for $1<p<2$.
			
			Now, let us prove (\ref{aeq019}) and (\ref{aeq020}). Let $h(T)=b(T)-a_{2}(T)+R_{1}(T)$, $T\in[0,\tau^{\star}]$. Then, we derive that
			\begin{align*}
				h(T)&=\int^{T}_{0}\mathcal{L}_{c}(h-R_{1})\textup{d}s+3\int^{T}_{0}\mathcal{F}_{c}(a_{1},a_{1},h-R_{1})\textup{d}s\\
				&~~~+\int^{T}_{0}G'_{c}(0)(h-R_{1})\textup{d}\tilde{W}(s).
			\end{align*}
			For $p\geq 2$, we use It$\hat{\textup{o}}^{,}$s formula again to  obtain that
			\begin{align*}
				&~~~~\|h(T)\|^{p}\\
				&\leq C\int^{T}_{0}\|h\|^{p-2}\langle\mathcal{L}_{c}(h-R_{1}),h\rangle\textup{d}s+C\int^{T}_{0}\|h\|^{p-2}\langle \mathcal{F}_{c}(a_{1},a_{1},h-R_{1}),h\rangle\textup{d}s\\
				&~~~+C\int^{T}_{0}\|h\|^{p-2}\langle h,\bar{G}'_{c}(0)(h-R_{1})\textup{d}\tilde{W}\rangle+C\int^{T}_{0}\|h\|^{p-2}\|h-R_{1}\|^{2}\textup{d}s.
			\end{align*}
			Thanks to condition (\ref{eq43}), Cauchy-Schwarz inequality and Young's inequality, we derive that
			\begin{align*}
				&~~~~\|h(T)\|^{p}\\
				&\leq C\int^{T}_{0}(\|h\|^{p}+\|R_{1}\|^{p}+\|R_{1}\|^{p}\|a_{1}\|^{2p})\textup{d}s
				+C\int^{T}_{0}\|h\|^{p-2}\|h-R_{1}\|^{2}\textup{d}s\\
				&~~~+C\int^{T}_{0}\|h\|^{p-2}\langle h,\bar{G}'_{c}(0)(h-R_{1})\textup{d}\tilde{W}\rangle.
			\end{align*}
			From Lemma \ref{ale010}, we know $R_{1}(T)=\mathcal{O}(\varepsilon^{1-9\kappa})$.
			Then, in virtue of  Lemma \ref{ale001} and Burkholder-Davis-Gundy inequality, it is easy to obtain that
			for any $T_{2}\in [0,T_{0}]$,
			\begin{align*}
				&~~~~\mathbb{E}(\sup_{0\leq T\leq T_{2}\wedge \tau^{\star}}\|h(T)\|^{p})\\
				&\leq C\varepsilon^{p-18\kappa p}(1+\|a_{1}(0)\|^{2p})+
				C\int^{T_{2}}_{0}\mathbb{E}(\sup_{0\leq s_{1}\leq s\wedge \tau^{\star}}\|h(s_{1})\|^{p})\textup{d}s.
			\end{align*}
			Furthermore, with the help of Gronwall's lemma, we own that
			\begin{align}
				\mathbb{E}\Big(\sup_{0\leq T\leq\tau^{\star}}\|h(T)\|^{p}\Big)\leq C\varepsilon^{p-18\kappa p}(1+\|a_{1}(0)\|^{2p}).\label{aeq021}
			\end{align}
			Therefore,  (\ref{aeq019}) and (\ref{aeq020}) are the consequent results of (\ref{aeq017}) and (\ref{aeq021}).
		\end{proof}
		
		\section{Main results}
		Recalling that there exists the term dependent of the parameter $\varepsilon$ in (\ref{aeq015}), thus we need to further deal with the trouble. Let us now investigate  (\ref{aeq015}) in two cases.

		{\bf Case \uppercase\expandafter{\romannumeral 1}}:
		$\int^{T}_{0}\varepsilon^{-1}G'_{c}(0)(\mathcal{Y})\textup{d}\tilde{W}=0,$ for $ T\in[0,T_{0}], a.s..$
		
		In this case,  we assume
		\begin{align}\label{eqe501}
			\int^{T}_{0}\varepsilon^{-1}G'_{c}(0)(\mathcal{Y})\textup{d}\tilde{W}=0, \text{for}~T\in[0,T_{0}], a.s..
		\end{align}
		We remark that the sufficient condition for (\ref{eqe501}) is  either $G_{s}'(0)\bar{n}\mathcal{U}=0$, for any $\bar{n}\in \mathcal{N}$ and   $\mathcal{U}\in U$, or $G_{c}'(0)\bar{s}\mathcal{U}=0$, for any $\bar{s}\in \mathcal{S}$ and   $\mathcal{U}\in U$.
		In this case, we can easily obtain  the second-order amplitude equations as follows:
		\begin{align}
			\bar{b}(T)&=a_{2}(0)+\int^{T}_{0}[\mathcal{L}_{c}\bar{b}+3\mathcal{F}_{c}(a_{1},a_{1},\bar{b})]\textup{d}s+\int^{T}_{0}[G_{c}'(0)(\bar{b})+\frac{1}{2}G_{c}''(0)(a_{1},a_{1})]\textup{d}\tilde{W}. \label{aeq016}
		\end{align}
		By Lemmas \ref{ale004} and \ref{ale021}, we conclude the following result.
		\begin{Lemma}\label{ale011}
			Let Assumptions \ref{assu1}-\ref{assu5} and condition $(\ref{eqe501})$ hold. For $p>1$, $\|a_{1}(0)\|_{\alpha}\leq \varepsilon^{-\frac{\kappa}{3}}$, $\|a_{2}(0)\|_{\alpha}\leq \varepsilon^{-\frac{\kappa}{3}}$, $\|\psi(0)\|_{\alpha}\leq \varepsilon$, there exists some constant $C>0$, such that
			\begin{equation*}
				\mathbb{E}\left(\sup_{0\leq T\leq \tau^{\star}}\|R_{2}(T)\|^{p}_{\alpha}\right)\leq C\varepsilon^{3p-19\kappa p},
			\end{equation*}
			where
			\begin{equation*}
				R_{2}(T)=u(\varepsilon^{-2}T)-\varepsilon a_{1}(T)-\varepsilon^{2} \bar{b}(T)-\varepsilon \mathcal{Q}(T)-\varepsilon K(T).
			\end{equation*}
		\end{Lemma}
		
		From Lemma \ref{ale011}, we obtain the second-order approximation of $u(t)$ for $t\in [0,\varepsilon^{-2}\tau^{\star}]$.  We are going to show  such approximation is still valid for  $t\in [0,\varepsilon^{-2}T_{0}]$ with high probability.
		For $\kappa>0$ , set $\Omega^{\star}\subset\Omega$ of all $\omega\subset\Omega$ such that all these estimates
		\begin{align*}
			&\sup_{0\leq T \leq \tau^{\star}}\|a_{1}(T)\|< \varepsilon^{-\frac{\kappa}{2}},~~
			\sup_{0\leq T \leq \tau^{\star}}\|a_{2}(T)\|< \varepsilon^{-\frac{5\kappa}{6}},\\
			&\sup_{0\leq T \leq \tau^{\star}}\|\psi(T)\|_{\alpha}<\varepsilon^{-\frac{\kappa}{2}},~~
			\sup_{0\leq T \leq \tau^{\star}}\|R_{2}(T)\|_{\alpha}< \varepsilon^{3p-20\kappa}.
		\end{align*}
		hold.
		\begin{Lemma}\label{ale019}
			Let Assumptions \ref{assu1}-\ref{assu5} and condition $(\ref{eqe501})$ hold. For $p>1$, $\|a_{1}(0)\|_{\alpha}\leq \varepsilon^{-\frac{\kappa}{3}}$, $\|a_{2}(0)\|_{\alpha}\leq \varepsilon^{-\frac{\kappa}{3}}$, $\|\psi(0)\|_{\alpha}\leq \varepsilon$,   there exists a constant $C>0$, such that
			\begin{equation*}
				\mathbb{P}(\Omega^{\star})\geq 1-C\varepsilon^{p}.
			\end{equation*}
		\end{Lemma}
		\begin{proof}
			For fixed $p>1$,  we obtain that
			\begin{align*}
				\mathbb{P}(\Omega^{\star})&\geq 1-\mathbb{P}(\sup_{0\leq T \leq \tau^{\star}}\|a_{1}(T)\|\geq\varepsilon^{-\frac{\kappa}{2}})-\mathbb{P}(\sup_{0\leq T \leq \tau^{\star}}\|a_{2}(T)\|\geq\varepsilon^{-\frac{5\kappa}{6}})\\
				&~~~-\mathbb{P}(\sup_{0\leq T \leq \tau^{\star}}\|\psi(T)\|_{\alpha}\geq\varepsilon^{-\frac{\kappa}{2}})-\mathbb{P}(\sup_{0\leq T \leq \tau^{\star}}\|R_{2}(T)\|_{\alpha}\geq\varepsilon^{3-20\kappa})\\
				&\geq 1-\varepsilon^{\frac{\kappa q}{2}}\mathbb{E}(\sup_{0\leq T \leq \tau^{\star}}\|a_{1}(T)\|^{q})-\varepsilon^{\frac{5\kappa q}{6}}\mathbb{E}(\sup_{0\leq T \leq \tau^{\star}}\|a_{2}(T)\|^{q})\\
				&~~~-\varepsilon^{\frac{\kappa q}{2}}\mathbb{E}(\sup_{0\leq T \leq \tau^{\star}}\|\psi(T)\|^{q}_{\alpha})-\varepsilon^{20\kappa q-3q}\mathbb{E}(\sup_{0\leq T \leq \tau^{\star}}\|R_{2}(T)\|^{q}_{\alpha})\\
				&\geq 1-C\varepsilon^{p},
			\end{align*}
			where we use Chebyshev inequality in the second inequality, and use Lemmas \ref{ale004}, \ref{ale001}, \ref{ale021} and \ref{ale011} with $q$ large enough in the last one.
		\end{proof}
		
		\begin{Theorem}\label{theo11}
		Let Assumptions \ref{assu1}-\ref{assu5} and condition $(\ref{eqe501})$ hold. For $p>1$, $\|a_{1}(0)\|_{\alpha}\leq \varepsilon^{-\frac{\kappa}{3}}$, $\|a_{2}(0)\|_{\alpha}\leq \varepsilon^{-\frac{\kappa}{3}}$, $\|\psi(0)\|_{\alpha}\leq \varepsilon$, and $\kappa\in(0,\frac{1}{20})$, there exists a constant $C>0$, such that
			\begin{equation*}
				\mathbb{P}\Big(\sup_{0\leq t\leq\varepsilon^{-2}T_{0}}\|u(t)-\varepsilon a_{1}(\varepsilon^{2}t)-\varepsilon^{2}\bar{b}(\varepsilon^{2}t)-\varepsilon \mathcal{Q}(\varepsilon^{2}t)-\varepsilon K(\varepsilon^{2}t)\|_{\alpha}\geq\varepsilon^{3-20\kappa}\Big)\leq C\varepsilon^{p}.
			\end{equation*}
		\end{Theorem}
		
		\begin{proof}
			Note that
			\begin{align*}
				\Omega^{\star}&\subseteq\Big\{\omega\Big|\sup_{0\leq T \leq \tau^{\star}}\|a_{1}(T)\|< \varepsilon^{-\kappa}, \sup_{0\leq T \leq \tau^{\star}}\|a_{2}(T)\|< \varepsilon^{-\kappa},
				\sup_{0\leq T \leq \tau^{\star}}\|\psi(T)\|_{\alpha}<\varepsilon^{-\kappa}\Big\}\\
				&\subseteq\{\omega\Big|\tau^{\star}=T_{0}\}\subseteq\Omega.
			\end{align*}
			Then, we have
			\begin{equation*}
				\sup_{0\leq T \leq T_{0}}\|R_{2}(T)\|_{\alpha}=\sup_{0\leq T \leq \tau^{\star}}\|R_{2}(T)\|_{\alpha}< \varepsilon^{3-20\kappa}, \omega\in \Omega^{\star}.
			\end{equation*}
			It follows from Lemma \ref{ale019} that
			\begin{equation*}
				\mathbb{P}(\sup_{0\leq T \leq T_{0}}\|R_{2}(T)\|_{\alpha}\geq\varepsilon^{3-20\kappa})\leq 1-\mathbb{P}(\Omega^{\star})\leq C\varepsilon^{p}.
			\end{equation*}
		\end{proof}
		
		{\bf Case \uppercase\expandafter{\romannumeral 2}}: $\int^{T}_{0}\varepsilon^{-1}G'_{c}(0)(\mathcal{Y})\textup{d}\tilde{W}\neq 0,$ for $T\in[0,T_{0}], a.s..$
		
		In this case, we consider
		 condition (\ref{eqe501})  can not hold, for $T\in[0,T_{0}]$, a.s..
		The
		term $\int^{T}_{0}\varepsilon^{-1}G'_{c}(0)(\mathcal{Y})\textup{d}\tilde{W}$  dependent  of the parameter $\varepsilon$  appears in (\ref{aeq015}), which  motivates us to pay more attention to such term. Our key approach is to separate the high-order terms from the quadratic variation of the diffusion term of (\ref{aeq015}). On the one hand, for the case that
		$1<\dim\mathcal{N}<\infty$,   the second-order amplitude equations can be obtained as desired by such approach, but we just can prove that the approximate solutions weakly converge the real solutions  without any precise rate. Since our motivation  is to improve the convergence rate between the approximate solutions and the real ones in the paper, we ignore the argument about the weak convergence. On the other hand,  for the case that $\dim\mathcal{N}=1$,  relying on Lemma 6.1 in \cite{Bl1}, we not only can obtain the  second-order amplitude equations, but also provide the rigorous convergence rate. Therefore, we focus on the case that  $\dim\mathcal{N}=1$ in the remainder of the section. As preparation, we provide the following notations and lemmas.

		Set
		\begin{align*}
			\bar{g}_{1}(T)&=\sum_{j=1}^{\infty}\int^{T}_{0}\langle G_{c}'(0)(b)f_{j},G_{c}'(0)(\mathcal{Y})f_{j}\rangle \textup{d}s,\\
			\bar{g}_{2}(T)&=\frac{1}{2}\sum_{j=1}^{\infty}\int^{T}_{0}\langle G_{c}''(0)(a_{1},a_{1})f_{j},G_{c}'(0)(\mathcal{Y})f_{j}\rangle\textup{d}s,\\
			\bar{g}_{3}(T)&=\mathcal{L}_{c}b(T)+3\mathcal{F}_{c}(a_{1}(T),a_{1}(T),b(T)),\\
			\bar{g}_{4}(T)&=G'_{c}(0)(b(T))+\frac{1}{2}G''_{c}(0)(a_{1}(T),a_{1}(T))+\varepsilon^{-1}G'_{c}(0)(\mathcal{Y}(T)).
		\end{align*}
		
		\begin{Lemma}\label{al29}
			Under Assumptions \ref{assu1}-\ref{assu5}, for $p>1$, $\|a_{1}(0)\|_{\alpha}\leq \varepsilon^{-\frac{\kappa}{3}}$, $\|a_{2}(0)\|_{\alpha}\leq \varepsilon^{-\frac{\kappa}{3}}$, $\|\psi(0)\|_{\alpha}\leq \varepsilon$, there exists some constant $C>0$, such that
			\begin{align}
				&\mathbb{E}\left(\sup_{0\leq T\leq T_{0}}|\bar{g}_{1}(T)|^{p}\right)\leq C\varepsilon^{2p-3\kappa p},\label{aeq029}\\
				&\mathbb{E}\left(\sup_{0\leq T\leq T_{0}}|\bar{g}_{2}(T)|^{p}\right)\leq C\varepsilon^{2p-4\kappa p},\label{aeq030}
			\end{align}
		\end{Lemma}
		\begin{proof}
			Since the proof of (\ref{aeq029}) and (\ref{aeq030}) are  similar, we just prove (\ref{aeq029}) here.
			Recall that $\mathcal{Y}(T)$ satisfies
			\begin{align*}
				\textup{d}\mathcal{Y}=\varepsilon^{-2}\mathcal{A}_{s}\mathcal{Y}\textup{d}T+	G_{s}'(0)(a_{1})\textup{d}\tilde{W}.
			\end{align*}
			Following the similar proof as in Lemma \ref{ale003}, we can obtain $\mathcal{Y}(T)= \mathcal{O}(\varepsilon^{1-\frac{4\kappa}{3}})$.
			
			Applying It$\hat{\textup{o}}^{,}$s formula to $\sum\limits_{j=1}^{\infty}\langle G_{c}'(0)(b)f_{j},G_{c}'(0)(\mathcal{A}_{s}^{-1}\mathcal{Y})f_{j}\rangle$, we have
			\begin{align*}
				&~~~~\varepsilon^{-2}\sum_{j=1}^{\infty}\int^{T}_{0}\langle G_{c}'(0)(b)f_{j},G_{c}'(0)(\mathcal{Y})f_{j}\rangle \textup{d}s\\
				&=\sum_{j=1}^{\infty}\langle G_{c}'(0)(b(T))f_{j},G_{c}'(0)(\mathcal{A}_{s}^{-1}\mathcal{Y}(T))f_{j}\rangle\\
				&~~~-\sum_{j=1}^{\infty}\int^{T}_{0}\langle G_{c}'(0)(b)f_{j},G_{c}'(0)(\mathcal{A}_{s}^{-1}(G'_{s}(0)(a_{1})\cdot))f_{j}\rangle \textup{d}\tilde{W}\\
				&~~~-\sum_{j=1}^{\infty}\int^{T}_{0}\langle G_{c}'(0)(   \bar{g}_{3}
				)f_{j},G_{c}'(0)(\mathcal{A}^{-1}_{s}\mathcal{Y})f_{j}\rangle \textup{d}s\\
				&~~~-
				\sum_{j=1}^{\infty}\int^{T}_{0}\langle G_{c}'(0)(\bar{g}_{4} \cdot)
				f_{j},G_{c}'(0)(\mathcal{A}^{-1}_{s}\mathcal{Y})f_{j}\rangle \textup{d}\tilde{W}\\
				&~~~-\sum_{j,l=1}^{\infty}\int^{T}_{0}\langle G_{c}'(0)(\bar{g}_{4}f_{l})
				f_{j},G_{c}'(0)(\mathcal{A}_{s}^{-1}G'_{s}(0)(a_{1})f_{l})f_{j}\rangle \textup{d}s\\
				&=:\sum^{5}_{i=1}I_{i}(T).
			\end{align*}
			Recalling $\mathcal{A}^{-1}_{s}$ is a bounded and linear operator,  it is easy to
			\begin{align*}
				&~~~~\mathbb{E}\Big(\sup_{0\leq T\leq T_{0}}	|I_{1}(T)|^{p}\Big)\\
				&\leq  \mathbb{E}\Big(\sup_{0\leq T\leq T_{0}}	\big(\sum_{j=1}^{\infty}\|G_{c}'(0)(b(T))f_{j}\|^{2}  \sum_{j=1}^{\infty}\|G_{c}'(0)(\mathcal{A}_{s}^{-1}\mathcal{Y}(T))f_{j}\|^{2}\big)^{\frac{p}{2}}\Big)\\
				&\leq C\mathbb{E}\Big(\sup_{0\leq T\leq T_{0}} \|b(T)\|^{2p} \Big)^{\frac{1}{2}} \mathbb{E}\Big(\sup_{0\leq T\leq T_{0}} \|\mathcal{Y}(T)\|_{\alpha}^{2p} \Big)^{\frac{1}{2}}\\
				&\leq C\varepsilon^{p-2\kappa p},
			\end{align*}
			where we use Cauchy-Schwarz inequality in the first inequality, use H\"{o}lder inequality and Assumption \ref{assu5} in the second one, and  use Lemma \ref{ale021} in the last one. By Burkholder-Davis-Gundy inequality, we can have
			\begin{align*}
				&~~~~\mathbb{E}\Big(\sup_{0\leq T\leq T_{0}}	\|I_{2}(T)\|^{p}\Big)\\
				&\leq C\mathbb{E}\Big( \int^{T_{0}}_{0}\|\sum_{j=1}^{\infty}\langle G_{c}'(0)(b)f_{j},G_{c}'(0)(\mathcal{A}_{s}^{-1}(G'_{s}(0)(a_{1})\cdot))f_{j}\rangle\|^{2}_{\mathscr{L}_{2}(U,\mathbb{R})}\textup{d}s\Big)^{\frac{p}{2}}\\
				&= C\mathbb{E}\Big( \int^{T_{0}}_{0}\sum_{l=1}^{\infty}|\sum_{j=1}^{\infty}\langle G_{c}'(0)(b)f_{j},G_{c}'(0)(\mathcal{A}_{s}^{-1}(G'_{s}(0)(a_{1})f_{l}))f_{j}\rangle|^{2}\textup{d}s\Big)^{\frac{p}{2}}\\
				&\leq  C \mathbb{E}\Big( \int^{T_{0}}_{0}\sum_{l=1}^{\infty} \big(\sum_{j=1}^{\infty}\| G_{c}'(0)(b)f_{j}\|^{2} \sum_{j=1}^{\infty}\| G_{c}'(0)(\mathcal{A}_{s}^{-1}(G'_{s}(0)(a_{1})f_{l}))f_{j}\|^{2}\big)\textup{d}s\Big)^{\frac{p}{2}}\\
				&\leq  C \mathbb{E}\Big( \int^{T_{0}}_{0}\sum_{l=1}^{\infty} \big(\|b\|^{2} \| \mathcal{A}_{s}^{-1}(G'_{s}(0)(a_{1})f_{l}))\|_{\alpha}^{2}\big)\textup{d}s\Big)^{\frac{p}{2}}\\
				&\leq  C \mathbb{E} \Big(\sup_{0\leq T\leq T_{0}} \|b\|^{p}\|a_{1}\|^{p}\Big)\\
				&\leq C\varepsilon^{-\kappa p}.
			\end{align*}
			Analogous arguments as above  can yield that
			\begin{align*}
				\sum_{i=3}^{5} \mathbb{E}\Big(\sup_{0\leq T\leq T_{0}} \|I_{i}(T)\|^{p} \Big)\leq C \varepsilon^{-3\kappa p}.
			\end{align*}
			Consequently, collecting all the estimations, we complete the proof.
		\end{proof}
		
		Here, we introduce some notations from \cite{Bl1}. For a Hilbert Space $\mathcal{H}$, denote the tensor product of it by $v_{1}\otimes v_{2}$, and the symmetric tensor product of it by $v_{1}\otimes_{s}v_{2}=\frac{1}{2}\big(v_{1}\otimes v_{2}+v_{2}\otimes v_{1}\big)$, where $v_{1}, v_{2}\in \mathcal{H}$. Furthermore, define the scalar product  in the
		tensor product space $\mathcal{H}^{\alpha}\otimes \mathcal{H}^{\beta}$ by $\langle u_{1}\otimes v_{1}, u_{2}\otimes v_{2} \rangle_{\alpha,\beta}:=\langle u_{1}, u_{2}\rangle_{\alpha}\langle v_{1}, v_{2} \rangle_{\beta}$, where $u_{1}, u_{2}\in \mathcal{H}^{\alpha}, v_{1}, v_{2}\in \mathcal{H}^{\beta}$. For simplicity of notation, we adopt $\langle\cdot,\cdot\rangle_{\alpha}$ instead of $\langle\cdot,\cdot\rangle_{\alpha,\alpha}$. The norm of $\mathcal{H}^{\alpha}\otimes\mathcal{H}^{\beta}$ is induced by such scalar product.
		For two linear operators $\mathcal{L}_{a}$ and $\mathcal{L}_{b}$ on $\mathcal{H}$, define the symmetric tensor product of them by $\big(\mathcal{L}_{a}\otimes_{s}\mathcal{L}_{b}\big)(v_{1}\otimes v_{2})=\frac{1}{2}
		\big(\mathcal{L}_{a} v_{1}\otimes \mathcal{L}_{b} v_{2}+\mathcal{L}_{b}v_{2}\otimes\mathcal{L}_{a}v_{1}\big)$.
		
		We note $\|v_{1}\otimes_{s}v_{2}\|_{\mathcal{H}^{\alpha}  \otimes_{s} \mathcal{H}^{\alpha}} \leq \|v_{1}\|_{\alpha}\|v_{2}\|_{\alpha}$, for any $v_{1},v_{2}\in\mathcal{H}^{\alpha}$ and $\alpha\in\mathbb{R}$.  It is also easy to check  $\textup{I}\otimes_{s} \mathcal{A}$ is an operator with eigenvalues $-\frac{\lambda_{k} + \lambda_{j}}{2}$ in the basis
		$e_{k}\otimes_{s}  e_{j}$. Then, we assert that $(\textup{I}\otimes_{s} \mathcal{A})^{-1}$ as pseudo-inverse  (i.e., $(\textup{I}\otimes_{s} \mathcal{A})^{-1}(\mathcal{N}\otimes_{s}\mathcal{N})=0$) is a  bounded operator from  $\mathcal{H}^{\alpha-\frac{1}{2}}\otimes_{s} \mathcal{H}^{\alpha-\frac{1}{2}}$ to $\mathcal{H}^{\alpha}\otimes_{s} \mathcal{H}^{\alpha}$, for $\alpha \in \mathbb{R}$. In fact, for $\mathcal{X}=\sum_{k,j=1}^{\infty} x_{kj}e_{k}\otimes_{s} e_{j}\in \mathcal{H}^{\alpha} \otimes_{s}
		\mathcal{H}^{\alpha}$, we have
		\begin{align*}
			\|(\textup{I}\otimes_{s} \mathcal{A})^{-1} \mathcal{X}\|_{\mathcal{H}^{\alpha} \otimes_{s}
				\mathcal{H}^{\alpha}}&=\Big(\sum_{\substack {k,j=1\\k=j\neq 1}}^{\infty} (\frac{2}{\lambda_{k} + \lambda_{j}})^{2} x_{kj}^{2} (\lambda_{k}+1)^{2\alpha}(\lambda_{j}+1)^{2\alpha}\Big)\\
			&\leq C\Big(\sum_{k,j=1}^{\infty}  x_{kj}^{2} (\lambda_{k}+1)^{2\alpha-1}(\lambda_{j}+1)^{2\alpha-1}\Big)\\
			&=C\|\mathcal{X}\|_{\mathcal{H}^{\alpha-\frac{1}{2}}\otimes_{s}\mathcal{H}^{\alpha-\frac{1}{2}}}.
		\end{align*}
		
		Define $\mathcal{G}(\cdot)$ be a linear and bounded operator from $\mathcal{H}^{\alpha}\otimes_{s} \mathcal{H}^{\alpha}$ to $\mathbb{R}$ by
		\begin{align*}
			\mathcal{G}(\mathcal{X}):=\sum_{k,j,l=1}^{\infty}x_{kj}\langle G_{c}'(0)(e_{k})f_{l},G_{c}'(0)(e_{j})f_{l} \rangle.
		\end{align*}
		Then $\mathcal{G}(\mathcal{Y}\otimes_{s}\mathcal{Y})=\sum_{l=1}^{\infty}\langle G_{c}'(0)(\mathcal{Y})f_{l},G_{c}'(0)(\mathcal{Y})f_{l} \rangle$.
		
		Moreover,  we need a notation to simplify the representation.
		For a family of real-valued processes $\{X_{\varepsilon}(T)\}_{T\geq0}$, we call $X_{\varepsilon}=\mathcal{O}(f_{\varepsilon})$ if for every $p>1$, there exists a positive constant $C$ such that
		\begin{equation*}
			\mathbb{E}\Big(\sup_{0\leq T\leq\tau^{\star}}|X_{\varepsilon}(T)|^{p}\Big)\leq Cf^{p}_{\varepsilon},
		\end{equation*}
		where $\tau^{\star}$ is given in Definition \ref{def1}.
		
		With the above convenient notations, we present the following lemma.
		\begin{Lemma}\label{al30}
			Under Assumptions \ref{assu1}-\ref{assu5}, for $p>1$, $\|a_{1}(0)\|_{\alpha}\leq \varepsilon^{-\frac{\kappa}{3}}$, $\|a_{2}(0)\|_{\alpha}\leq \varepsilon^{-\frac{\kappa}{3}}$, $\|\psi(0)\|_{\alpha}\leq \varepsilon$,  we have
			\begin{align*}
				\varepsilon^{-2}\int^{T}_{0}\mathcal{G}(\mathcal{Y}\otimes_{s}\mathcal{Y})\textup{d}s&=-\frac{1}{2}\sum_{l=1}^{\infty}\int^{T}_{0}\mathcal{G}(\textup{I}\otimes_{s}\mathcal{A})^{-1}(G'_{s}(0)(b)f_{l}\otimes_{s}G'_{s}(0)(b)f_{l})\textup{d}s \\
				&~~~~+\mathcal{O}(\varepsilon^{1-3\kappa}).
			\end{align*}
			In particular, we have for $T\in [0,T_{0}]$,
			\begin{align}
				-\frac{1}{2}\sum_{l=1}^{\infty}\mathcal{G}(\textup{I}\otimes_{s}\mathcal{A})^{-1}(G'_{s}(0)(b)f_{l}\otimes_{s}G'_{s}(0)(b)f_{l})(T)\geq 0,~~a.s..\label{aeq050}
			\end{align}
			
		\end{Lemma}
		\begin{proof}
			Applying It$\hat{\textup{o}}^{,}$s formula to $\mathcal{Y}\otimes \mathcal{Y}$, we obtain
			\begin{align}
				\varepsilon^{-2}\int^{T}_{0}(\mathcal{Y}\otimes_{s}\mathcal{A}\mathcal{Y})\textup{d}s&=\frac{1}{2}(\mathcal{Y}\otimes_{s} \mathcal{Y})(T)-\int^{T}_{0}\mathcal{Y}\otimes_{s} G_{s}'(0)(b)\textup{d}\tilde{W}\notag\\
				&~~~-\frac{1}{2}\int^{T}_{0}[\sum_{l=1}^{\infty}G_{s}'(0)(b) f_{l}\otimes_{s}G_{s}'(0)(b) f_{l}]\textup{d}s.\label{aeq031}
			\end{align}
			Noticing $\mathcal{Y}\otimes_{s}\mathcal{A}\mathcal{Y}\in \mathcal{H}^{\alpha-\frac{1}{2}}\otimes_{s}\mathcal{H}^{\alpha-\frac{1}{2}}$, we then apply $\mathcal{G}(\textup{I}\otimes_{s}\mathcal{A})^{-1}$ to both sides of (\ref{aeq031}) to have
			\begin{align*}
				&~~~~~~\mathbb{E}\Big(\sup_{0\leq T\leq T_{0}}|\varepsilon^{-2}\int^{T}_{0}\mathcal{G}(\mathcal{Y}\otimes_{s}\mathcal{Y})\textup{d}s\\
				&~~+\frac{1}{2}\sum_{l=1}^{\infty}\int^{T}_{0}\mathcal{G}(\textup{I}\otimes_{s}\mathcal{A})^{-1}(G'_{s}(0)(b)f_{l}\otimes_{s}G'_{s}(0)(b)f_{l})\textup{d}s
				|^{p}\Big)\\
				&\leq C \mathbb{E}\Big(\sup_{0\leq T\leq T_{0}}| \mathcal{G}(\textup{I}\otimes_{s}\mathcal{A})^{-1}(\mathcal{Y}\otimes_{s}\mathcal{Y})(T)|^{p}\Big)\\
				&~~+ C \mathbb{E}\Big(\sup_{0\leq T\leq T_{0}}|\sum_{l=1}^{\infty}\int^{T}_{0} \mathcal{G}(\textup{I}\otimes_{s}\mathcal{A})^{-1}(\mathcal{Y}\otimes_{s}G'_{s}(0)(b)f_{l})\textup{d}\beta_{l}|^{p}\Big)\\
				&\leq C\varepsilon^{p-3\kappa p},
			\end{align*}
			where we use $\mathcal{Y}(T)= \mathcal{O}(\varepsilon^{1-\frac{4\kappa}{3}})$, $b(T)= \mathcal{O}(\varepsilon^{-\frac{2\kappa}{3}})$ and Burkholder-Davis-Gundy inequality in the last inequality.
			
			As for the proof of (\ref{aeq050}), it is sufficient to prove $\bar{\mathcal{M}}=
			\sum_{k,j=2}^{\infty}\frac{1}{\lambda_{k}+\lambda_{j}}m_{k}m_{j}\geq 0$, for any $m_{k}\in\mathbb{R}$. Let $\mathcal{P}_{1}=(e^{-\lambda_{2}x},\cdots,e^{-\lambda_{n}x},\cdots)$, $\mathcal{P}_{2}=(m_{2},\cdots,m_{n},\cdots)$. Notice that $\mathcal{P}_{2}\mathcal{P}_{1}^{T}	\mathcal{P}_{1}\mathcal{P}_{2}^{T}$ is non-negative. Then, we have
			\begin{align*}
				\bar{\mathcal{M}}=\int^{\infty}_{0}\mathcal{P}_{2}\mathcal{P}_{1}^{T}	\mathcal{P}_{1}\mathcal{P}_{2}^{T}\textup{d}x \geq 0.
			\end{align*}
		\end{proof}
		
		The second-order amplitude equation considered here is just one-dimensional ODE, so we introduce the following notations for the sake of convenience:
		\begin{align}
			\tilde{a}_{2}(0)&=\langle a_{2}(0), e_{1}\rangle\notag\\
			\tilde{b}_{1}(T)&=\langle b(T), e_{1}\rangle,\notag\\
			\sigma_{1}&=\sum_{j=1}^{\infty}\langle G_{c}'(0)(e_{1})f_{j}, G_{c}'(0)(e_{1})f_{j}\rangle\notag\\
			&~~~-\frac{1}{2}\mathcal{G}(\textup{I}\otimes_{s}\mathcal{A})^{-1}(G'_{s}(0)(e_{1})f_{j}\otimes_{s}G'_{s}(0)(e_{1})f_{j})\notag\\
			\sigma_{2}(T)&=\sum_{j=1}^{\infty}\langle G_{c}'(0)(e_{1})f_{j}, G_{c}''(0)(a_{1},a_{1})f_{j}\rangle,\notag\\
			\sigma_{3}(T)&=\frac{1}{4}\sum_{j=1}^{\infty}\langle G_{c}''(0)(a_{1},a_{1})f_{j},G_{c}''(0)(a_{1},a_{1})f_{j}\rangle,\notag\\
			\sigma_{4}(T)&=\langle \mathcal{L}_{c}e_{1}+3\mathcal{F}_{c}(a_{1}(T),a_{1}(T),e_{1}),e_{1}\rangle\notag\\
			\bar{g}_{5}(T)&=\int^{T}_{0}\|G_{c}'(0)(b)+\frac{1}{2}G_{c}''(0)(a_{1},a_{1})+\varepsilon^{-1}G'_{c}(0)(\mathcal{Y})\|^{2}_{\mathscr{L}_{2} (U,\mathcal{H}^{\alpha})}\textup{d}s,\label{aeq090}\\
			\bar{g}_{6}(T)&=\int^{T}_{0}[\sigma_{1}\tilde{b}_{1}^{2}+\sigma_{2}(s)\tilde{b}_{1}+\sigma_{3}(s)]\textup{d}s,\label{aeq091}\\
			M_{1}(T)&=\langle\int^{T}_{0}[G_{c}'(0)(b)+\frac{1}{2} G_{c}''(0)(a_{1},a_{1})+\varepsilon^{-1}G'_{c}(0)(\mathcal{Y})]\textup{d}\tilde{W},e_{1}\rangle.\notag
		\end{align}
		
		Then, Lemmas \ref{al29} and \ref{al30} immediately yield the following lemma .
		\begin{Lemma}\label{al31}
			Under Assumptions \ref{assu1}-\ref{assu5}, for $p>1$, $\|a_{1}(0)\|_{\alpha}\leq \varepsilon^{-\frac{\kappa}{3}}$, $\|a_{2}(0)\|_{\alpha}\leq \varepsilon^{-\frac{\kappa}{3}}$, $\|\psi(0)\|_{\alpha}\leq \varepsilon$, we have
			\begin{align*}
				\bar{g}_{5}(T)=\bar{g}_{6}(T)+\mathcal{O}(\varepsilon^{1-3\kappa}).
			\end{align*}
		\end{Lemma}
		
		The next lemma plays an important role in the process of dealing with the high-order terms of the diffusion part.
		\begin{Lemma}\label{al20}
			Let $\bar{M}_{1}(T)$ be a continuous martingale with respect to some filtration $(\mathscr{F}_{T})_{T\geq0}$. Denote the quadratic variation of $\bar{M}_{1}(T)$ by $\bar{f}_{1}(T)$ and let $\bar{f}_{2}(T)$ be an arbitrary $\mathscr{F}_{T}$-adapted increasing process with
			$\bar{f}_{2}(0)=0$. Then, for $p>1$, there exists a filtration $\tilde{\mathscr{F}}_{T}$ with $\mathscr{F}_{T}\subset\tilde{\mathscr{F}}_{T}$ and a continuous $\tilde{\mathscr{F}}_{T}$ martingale $M_{2}(T)$ with quadratic
			variation $\bar{f}_{2}(T)$ such that, for every $r_{0}<\frac{1}{2}$, there exists a constant $C$ with
			\begin{equation*}
				\begin{split}
					\mathbb{E}\sup_{0\leq T\leq T_{0}}\Big|\bar{M}_{1}(T)-\bar{M}_{2}(T)\Big|^{p}&\leq C (\mathbb{E}|\bar{f}_{2}(T_{0})|^{2p})^{\frac{1}{4}}\Big(\mathbb{E}\sup_{0\leq T\leq T_{0}}|\bar{f}_{1}(T)-\bar{f}_{2}(T)|^{p}\Big)^{r_{0}}\\
					&\quad+C\mathbb{E}\sup_{0\leq T\leq T_{0}}|\bar{f}_{1}(T)-\bar{f}_{2}(T)|^{\frac{p}{2}}.
				\end{split}
			\end{equation*}
		\end{Lemma}
		
		Noting $M_{1}(T)$ is a continuous real-valued martingale with the quadratic variation $\bar{g}_{5}(T)$ given in (\ref{aeq090}), then we have the following lemma in view of Lemma \ref{al20}.
		\begin{Lemma}\label{al021}
			Under Assumptions \ref{assu1}-\ref{assu5}, for $p>1$, $\|a_{1}(0)\|_{\alpha}\leq \varepsilon^{-\frac{\kappa}{3}}$, $\|a_{2}(0)\|_{\alpha}\leq \varepsilon^{-\frac{\kappa}{3}}$, $\|\psi(0)\|_{\alpha}\leq \varepsilon$, there exists a continuous
			$\tilde{\mathscr{F}}_{T}$ martingale $M_{2}(T)$ with the quadratic variation $\bar{g}_{6}(T)$ given (\ref{aeq091}),  a constant $C>0$ and $\tilde{\gamma}\in(0,\frac{1}{2})$, such that
			\begin{equation}\label{aeq033}
				\mathbb{E}\Big(\sup_{0\leq T \leq T_{0}}|M_{1}(T)-M_{2}(T)|^{p}\Big)\leq C \varepsilon^{\tilde{\gamma} p-3\kappa p}.
			\end{equation}
			Moreover, there exists a Brownian motion $B(T)$ with respect to the filtration $\tilde{\mathscr{F}}_{T}$, such that
			\begin{equation}\label{aeq034}
				M_{2}(T)=\int^{T}_{0}[\sigma_{1}\tilde{b}^{2}_{1}+\sigma_{2}(s)\tilde{b}_{1}+\sigma_{3}(s) ]^{\frac{1}{2}}\textup{d}B.
			\end{equation}
		\end{Lemma}
		\begin{proof}
			Thanks to Lemmas \ref{ale010} and \ref{ale021}, we obtain
			\begin{align}\label{aeq093}
				\mathbb{E}\Big(\sup_{0\leq T \leq T_{0}}|\bar{g}_{5}(T)|^{p}\Big)\leq C\varepsilon^{-\frac{4\kappa p}{3}}.
			\end{align}
			Then, Lemma \ref{al20} and (\ref{aeq093})  imply (\ref{aeq033}). And, applying martingale representation theorem to $M_{2}(T)$, we derive (\ref{aeq034}).
		\end{proof}
		
		In view of Lemma \ref{al021}, we can get rid of the high-order terms hiding in $M_{1}$, and obtain the second-order amplitude equation as follows:
		\begin{align}
			\tilde{b}_{2}(T)=\tilde{a}_{2}(0)+\int^{T}_{0}\sigma_{4}(s)\tilde{b}_{2}\textup{d}s+\int^{T}_{0}[\sigma_{1}\tilde{b}^{2}_{2}+\sigma_{2}(s)\tilde{b}_{2}+\sigma_{3}(s) ]^{\frac{1}{2}}\textup{d}B.\label{aeq095}
		\end{align}
		Here, we need to note that (\ref{aeq095}) is well defined. Indeed,  $\sigma_{2}(t)$ and $\sigma_{3}(t)$ are random processes with respect to the filtration $\mathscr{F}_{t}$, while $B(T)$ is a Brownian motion with respect to the filtration $\tilde{\mathscr{F}}_{T}$ containing $\mathscr{F}_{T}$. Thus, it is reasonable to estimate
		the error between $\tilde{b}_{1}(T)$ and $\tilde{b}_{2}(T)$ on the stochastic base $(\Omega,\mathscr{F}, \{ \tilde{\mathscr{F}}_{t}\}_{t\geq 0},\mathbb{P})$. In order to achieve this goal, an auxiliary equation is considered as follows:
		\begin{align*}
			\tilde{b}_{3}(T)=\tilde{a}_{2}(0)+\int^{T}_{0}\sigma_{4}(s)\tilde{b}_{3}\textup{d}s+\int^{T}_{0}[\sigma_{1}\tilde{b}^{2}_{1}+\sigma_{2}(s)\tilde{b}_{1}+\sigma_{3}(s) ]^{\frac{1}{2}}\textup{d}B.
		\end{align*}
		Let us present  $\tilde{b}_{1}(T)$ can be approximated by $\tilde{b}_{3}(T)$ well  as the following lemma.
		\begin{Lemma}\label{ale101}
			Under Assumptions \ref{assu1}-\ref{assu5}, for $T_{0}>0$, $p>1$, $\|a_{1}(0)\|_{\alpha}\leq \varepsilon^{-\frac{\kappa}{3}}$, $\|a_{2}(0)\|_{\alpha}\leq \varepsilon^{-\frac{\kappa}{3}}$, $\|\psi(0)\|_{\alpha}\leq \varepsilon$, we have
			\begin{align}
				&\mathbb{E}\Big(\sup_{0\leq T\leq T_{0}}|\tilde{b}_{3}(T)|^{p}\Big)\leq C\varepsilon^{-\frac{4\kappa p}{3}}, \label{aeq102}\\
				&\mathbb{E}\Big(\sup_{0\leq T\leq T_{0}}|\tilde{b}_{1}(T)-\tilde{b}_{3}(T)|^{p}\Big)\leq C\varepsilon^{\tilde{r}p-\frac{20 \kappa p}{3}}.\label{aeq103}
			\end{align}
		\end{Lemma}
		\begin{proof}
			This proof is similar to that of Lemma \ref{ale021}. In fact, applying It$\hat{\textup{o}}^{,}$s formula to $\tilde{b}_{3}(T)$, we derive for $p\geq 2$:
			\begin{align*}
				\tilde{b}^{p}_{3}(T)&=\tilde{a}^{p}_{2}(0)+ \int^{T}_{0}[p\sigma_{4}(s)\tilde{b}_{3}^{p}+\frac{p(p-1)}{2}(\sigma_{1}\tilde{b}_{1}^{2}+\sigma_{2}(s)\tilde{b}_{1}+\sigma_{3}(s))\tilde{b}_{3}^{p-2}]\textup{d}T\\
				&~~~+\int^{T}_{0}p(\sigma_{1}\tilde{b}_{1}^{2}+\sigma_{2}(s)\tilde{b}_{1}+\sigma_{3}(s))^{\frac{1}{2}}\tilde{b}_{3}^{p-1}\textup{d}B.
			\end{align*}
			By Lemmas  \ref{ale001} and \ref{ale021}, we know $a_{1}(T)=\mathcal{O}(\varepsilon^{-\frac{\kappa}{3}})$ and $b_{1}(T)=\mathcal{O}(\varepsilon^{-\frac{2\kappa}{3}})$. Then,  performing the similar deductions provided in the proof of Lemma \ref{ale021}, it is easy to obtain (\ref{aeq102}). And, following from $M_{1}(T)-M_{2}(T)=\mathcal{O}(\varepsilon^{\tilde{r}-3\kappa})$,  (\ref{aeq103})  can be obtained immediately.
		\end{proof}
		
		Next lemma is devoted to  the error between   $\tilde{b}_{2}(T)$ and $\tilde{b}_{3}(T)$.
		\begin{Lemma}\label{ale201}
			Under Assumptions \ref{assu1}-\ref{assu5}, for $T_{0}>0$, $p>1$, $\|a_{1}(0)\|_{\alpha}\leq \varepsilon^{-\frac{\kappa}{3}}$, $\|a_{2}(0)\|_{\alpha}\leq \varepsilon^{-\frac{\kappa}{3}}$, $\|\psi(0)\|_{\alpha}\leq \varepsilon$, we have
			\begin{align}
				&\mathbb{E}\Big(\sup_{0\leq T\leq T_{0}}|\tilde{b}_{2}(T)|^{p}\Big)\leq C\varepsilon^{-\frac{2\kappa p}{3}}, \label{aeq105}\\
				&\mathbb{E}\Big(\sup_{0\leq T\leq T_{0}}|\tilde{b}_{2}(T)-\tilde{b}_{3}(T)|^{p}\Big)\leq C\varepsilon^{\tilde{r}p-\frac{20 \kappa p}{3}}.\label{aeq106}
			\end{align}
		\end{Lemma}
		\begin{proof}
			Noting $\sigma_{2}(T)=\mathcal{O}(\varepsilon^{-2\kappa})$ and  $\sigma_{3}(T)=\mathcal{O}(\varepsilon^{-4\kappa})$, we can obtain (\ref{aeq015}) as the analogous arguments in the proof of Lemma \ref{ale021}. Let us move to prove (\ref{aeq106}). Setting
			$\tilde{b}_{4}(T)=\tilde{b}_{2}(T)-\tilde{b}_{3}(T)$, we have
			\begin{align*}
				\tilde{b}_{4}(T)&=\int^{T}_{0}\sigma_{4}(s)\tilde{b}_{4}\textup{d}s+\int^{T}_{0}(\bar{g}_{7}(s)-\bar{g}_{8}(s))\textup{d}B,
			\end{align*}
			where $\bar{g}_{7}(s)= [\sigma_{1}\tilde{b}_{2}^{2}+\sigma_{2}(s)\tilde{b}_{2}+\sigma_{3}(s)]^{\frac{1}{2}}$
			and $\bar{g}_{8}(s)= [\sigma_{1}\tilde{b}_{1}^{2}+\sigma_{2}(s)\tilde{b}_{1}+\sigma_{3}(s)]^{\frac{1}{2}}$.
			Note that there exists a positive constant $C$  such that  for any $T\in [0,T_{0}]$,  ,
			\begin{align}
				|\bar{g}_{7}(T)-\bar{g}_{8}(T)|\leq C |\tilde{b}_{2}(T)-\tilde{b}_{1}(T)|,~~a.s..\label{aeq126}
			\end{align}
			
			For $p>2$, It$\hat{\textup{o}}^{,}$s formula yields that
			\begin{align}
				|\tilde{b}_{4}(T)|^{p}&\leq C \int^{T}_{0}|\tilde{b}_{4}|^{p}\textup{d}s+C\int^{T}_{0}|\tilde{b}_{4}|^{p-2}|\bar{g}_{7}(s)-\bar{g}_{8}(s)|^{2}\textup{d}s\notag\\
				&~~~+C\int^{T}_{0}|\tilde{b}_{4}|^{p-1}|\bar{g}_{7}(s)-\bar{g}_{8}(s)|\textup{d}B.\label{aeq129}
			\end{align}
			Taking the expectations of both sides of (\ref{aeq129}) and using the property (\ref{aeq126}), we derive
			\begin{align*}
				\mathbb{E}\Big(|\tilde{b}_{4}(T)|^{p}\Big)\leq C\int^{T}_{0}\mathbb{E}\Big(|\tilde{b}_{4}|^{p}\Big)\textup{d}s+C\int^{T}_{0}\mathbb{E}\Big(|\tilde{b}_{1}-\tilde{b}_{3}|^{p}\Big)\textup{d}s.
			\end{align*}
			Then, by Gronwall's lemma and (\ref{aeq103}), it is easy to obtain
			\begin{align}
				\mathbb{E}\Big(|\tilde{b}_{4}(T)|^{p}\Big)\leq C  \varepsilon^{\tilde{r}p-\frac{20 \kappa p}{3}},~~T\in[0,T_{0}].\label{aeq133}
			\end{align}
			Next, let us return to (\ref{aeq129}) and estimate the expectation of the supremum of $|\tilde{b}_{4}(T)|^{p}$.
			Using Burkholder-Davis-Gundy inequality and Young  inequality, we can have
			\begin{align*}
				\mathbb{E}\Big(\sup_{0\leq T\leq T_{0}}|\tilde{b}_{4}(T)|^{p}\Big)\leq C\mathbb{E}\Big(\int^{T_{0}}_{0}|\tilde{b}_{4}|^{p}\textup{d}s\Big)+C\mathbb{E}\Big(\int^{T_{0}}_{0}|\tilde{b}_{1}-\tilde{b}_{3}|^{2p}\textup{d}s\Big)^{\frac{1}{2}}.
			\end{align*}
			Consequently, it follows from (\ref{aeq103}) and (\ref{aeq133}) that (\ref{aeq106}) holds.
		\end{proof}
		
		Obviously, Lemmas \ref{ale101} and \ref{ale201}  yield that  $\tilde{b}_{1}(T)=\tilde{b}_{2}(T)+\mathcal{O}(\varepsilon^{\tilde{r}p-\frac{20 \kappa p}{3}}).$
		Then, combining with Lemmas \ref{ale004}, \ref{ale021},  we present the main result as follows.
		\begin{Theorem}\label{the005}
			Under Assumptions \ref{assu1}-\ref{assu5}, for $T_{0}>0$, $p>1$, $\|a_{1}(0)\|_{\alpha}\leq \varepsilon^{-\frac{\kappa}{3}}$, $\|a_{2}(0)\|_{\alpha}\leq \varepsilon^{-\frac{\kappa}{3}}$, $\|\psi(0)\|_{\alpha}\leq \varepsilon$, and $\kappa\in(0,\frac{1}{20})$, there exists a constant $C>0$, such that
			\begin{equation*}
				\mathbb{P}\Big(\sup_{0\leq t\leq\varepsilon^{-2}T_{0}}\|u(t)-\varepsilon a_{1}(\varepsilon^{2}t)-\varepsilon^{2}\tilde{b}_{3}(\varepsilon^{2}t)e_{1}-\varepsilon \mathcal{Q}(\varepsilon^{2}t)-\varepsilon K(\varepsilon^{2}t)\|_{\alpha}\geq\varepsilon^{2+\tilde{r}-20\kappa}\Big)\leq C\varepsilon^{p}.
			\end{equation*}
		\end{Theorem}
		\section{Application}
		In this section, we are going to apply our results to stochastic Allen-Cahn equation, and highlight that the approximate solution  constructed by the second-order amplitude equation is more accurate via numerical stimulation.
		
		{\bf Stochastic Allen-Cahn Equation}
		
		Consider stochastic Allen-Cahn equation with multiplicative noise written as
		\begin{align}
			\partial_{t}u&=(\partial_{xx}+1)u+\varepsilon^{2}u-u^{3}+\varepsilon[\sin{u}-\cos{u}+\cos{hu}]\cdot\partial_{t}W(t),\label{eqe01}\\
			u(0)&=u_{0},\notag
		\end{align}
		with respect to Dirichlet boundary conditions on $[0,\pi]$, where $W(t)$ is Wiener process, and $h$ is constant.
		
		Set $\mathcal{H}:=L^{2}[0,\pi]$, $\mathcal{A}:=\partial_{xx}+1$, $\mathcal{L}:=I$ and $\mathcal{F}(u):=-u^{3}$, $G(u)=\sin{u}-\cos{u}+\cos{hu}$.
		
		Firstly, let us consider {\bf Case \uppercase\expandafter{\romannumeral 1}} and check each Assumption as follows.
		
		Note $-\mathcal{A}e_{k}(x)=\lambda_{k}e_{k}(x)$ with $\lambda_{k}=k^{2}-1$ and $e_{k}(x)=\sqrt{\frac{2}{\pi}}\sin{kx}$, where $e_{k}(x)$
		is an orthonormal basis of $\mathcal{H}$.  Then,  Assumption \ref{assu1} holds with $m=2$ and kernel space $\mathcal{N}=\{e_{1}\}$.
		
		The function space considered here is $\mathcal{H}^{1}$ as defined in Definition 2.1.
		Obviously, Assumption $\ref{assu2}$ is true with $\alpha=1$ and $\beta=0$.
		
		Noticing that $\mathcal{H}^{1}$ is Banach algebra,    condition $(\ref{eq41})$ is satisfied.  And,
		we can easily obtain $-\int^{\pi}_{0}\sin(x)^{4}\textup{d}x<0$, so
		condition (\ref{eq43}) is valid. In addition,  since $\mathcal{F}$ is a standard cubic nonlinearity,  condition (\ref{eq44}) is also true . Thus, Assumption $\ref{assu3}$  holds.
		
		We suppose that $W(t)$ is one-dimensional real-valued Brownian motion on a stochastic base
		$(\Omega,\mathscr{F},\{\mathscr{F}_{t}\}_{t\geq 0},\mathbb{P})$. Obviously, Assumption $\ref{assu4}$ is satisfied.
		
		It is easy to  check that $G(\cdot)$ is  an Hilbert-Schmidt operator from $\mathcal{H}^{1}$ to $\mathscr{L}_{2}(\mathcal{H}^{1},\mathbb{R})$, and  satisfy that for any $u_{1} ,u_{2}, u_{3}\in\mathcal{H}$, $v\in\mathbb{R}$:
		\begin{align*}
			&G'(0)(u_{1})\cdot v=u_{1}\cdot v,\\
			& G''(0)(u_{1},u_{2})\cdot v=-(h^2-1)u_{1}u_{2}\cdot v,\\
			&G'''(0)(u_{1},u_{2},u_{3})\cdot v=u_{1}u_{2}u_{3}\cdot v.
		\end{align*}
		All conditions of Assumption \ref{assu5} are satisfied.
		
		Denote $\mathcal{S}$  the orthogonal complement of by $\mathcal{N}$.  For any $u_{1}\in\mathcal{N}$, $u_{2}\in\mathcal{S}$, $v\in\mathbb{R}$, we have
		\begin{align*}
			G'_{s}(0)(u_{1})\cdot v_{1}=0,\\
			G'_{c}(0)(u_{2})\cdot v_{1}=0,
		\end{align*}
	which imply that condition (\ref{eqe501}) holds.
	
		Therefore, we can apply Theorem \ref{theo11} to (\ref{eqe01}), and derive the first-order amplitude equation:
		\begin{align*}
			a_{1}(T)=a_{1}(0)+\int^{T}_{0}(a_{1}(s)-\frac{3}{2\pi}a_{1}^{3}(s))\textup{d}s+\int^{T}_{0} a_{1}(s)\textup{d}\tilde{W}(s),
		\end{align*}
		
		and second amplitude equation:
		\begin{align*}
			a_{2}(T)=\int^{T}_{0}(a_{2}(s)-\frac{9}{2\pi}a_{2}(s)a_{1}^{2}(s))\textup{d}s+\int^{T}_{0}( a_{2}(s)-\frac{4(h^2-1)\sqrt{2}}{3\pi^{\frac{3}{2}}}a^{2}_{1}(s))\textup{d}\tilde{W}(s),
		\end{align*}
		where $T=\varepsilon^{2}t$, $\tilde{W}(T)$ is the rescaled version of $W(t)$, $a_{1}(0)=\frac{P_{c}u_{0}}{\varepsilon}.$
		
		Assuming $u_{0}=\varepsilon \sqrt\frac{2}{\pi}\sin(x)$,  then we obtain
		\begin{align}
			u(t)&=\varepsilon a_{1}(\varepsilon^{2}t)\sqrt{\frac{2}{\pi}}\sin(x)+\mathcal{O}(\varepsilon^2), \label{eqe101} \\
			u(t)&=\varepsilon a_{1}(\varepsilon^{2}t)\sqrt{\frac{2}{\pi}}\sin(x)+\varepsilon^{2} a_{2}(\varepsilon^{2}t)\sqrt{\frac{2}{\pi}}\sin(x)+\mathcal{O}(\varepsilon^3),
			\label{eqe301}
		\end{align}
		where (\ref{eqe101}) follows from Theorem 2.1 in \cite{Fu1}, and (\ref{eqe301}) holds due to Theorem \ref{theo11}.
		
		In the followings, let us compare (\ref{eqe101}) and (\ref{eqe301}) via numerical analysis. We use  drift-implicit Euler–spectral Galerkin scheme \cite{Liuzh} to compute the numerical solution of $u$. And, the numerical solutions of $a_{1}$ and $a_{2}$ are solved by Euler method \cite{Hig,Kloe,Mao}. Then, consider  error $\bar{R}_{1}(t):=\|u(t)-\varepsilon a_{1}(\varepsilon^{2}t)\sqrt{\frac{2}{\pi}}\sin(x)\|_{L^{2}}$, and error $\bar{R}_{2}(t):=\|u(t)-\varepsilon a_{1}(\varepsilon^{2}t)\sqrt{\frac{2}{\pi}}\sin(x)-\varepsilon^{2} a_{2}(\varepsilon^{2}t)\sqrt{\frac{2}{\pi}}\sin(x)\|_{L^{2}}$. Here, we want to explain that the numerical solution of $u$ is valid in the sense of $L^{2}$ (please see Theorem 4.1 in \cite{Liuzh}), so  the errors  are considered in  the sense of $L^{2}$ rather than $H^{1}$. Nevertheless, the errors are still reliable due to embedding theorem.
		
		In Figure 1, we  consider  $\varepsilon=0.01, h=20$ and  $\varepsilon=0.001, h=100 $ respectively, and show the mean  of the error of 100 samples in the sense of $L^2$.
		
		\begin{figure}[h]
			\subfigure[$\varepsilon=0.01,h=20$]{
				\hspace{-0.5cm}
				\includegraphics[width=8.5cm]{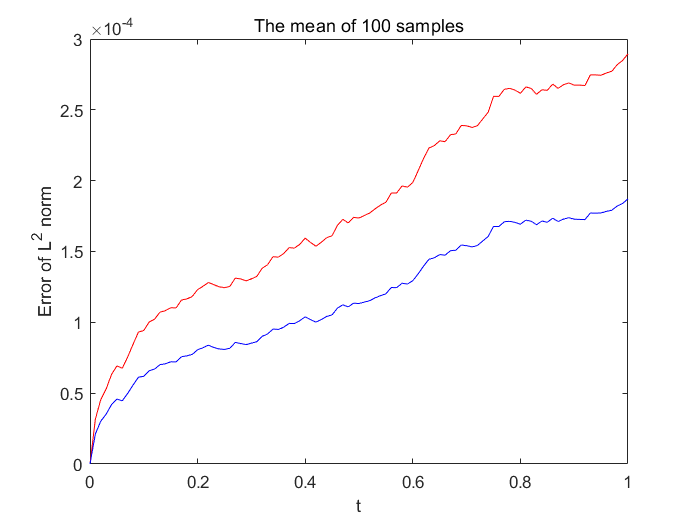}
				\label{eqf001}
			}
			\hspace{0.000005cm}
			\subfigure[$\varepsilon=0.001,h=100$]
			{\centering
				\includegraphics[width=8.5cm]{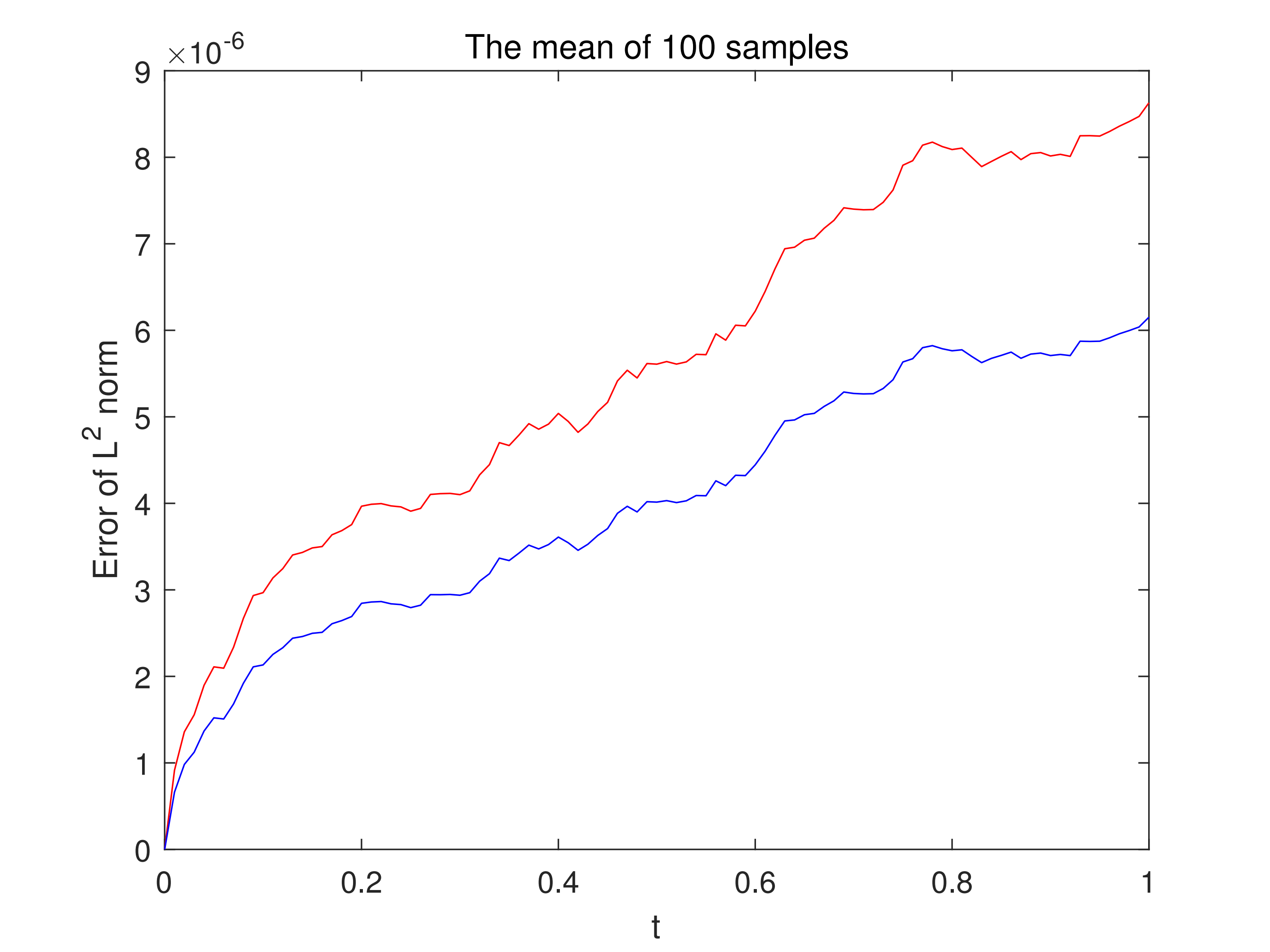}
				\label{eqf002}}
			\caption{Red and blue curves represent the time evolution  of the error $\bar{R}_{1}(t)$ and  $\bar{R}_{2}(t)$, respectively.}
			\label{eqf003}
		\end{figure}
		
		Through the comparisons, it is intuitive to illustrate  that the approximation solution shown in (\ref{eqe301}) is more accurate.
		
		We proceed to consider {\bf Case \uppercase\expandafter{\romannumeral 2}}. Here, we suppose that $W(t)=\sqrt{\frac{2}{\pi}}\sin{kx}\beta_{t}$, where $\beta(t)$ is one-dimensional real-valued Brownian motion on a stochastic base
		$(\Omega,\mathscr{F},\{\mathscr{F}_{t}\}_{t\geq 0},\mathbb{P})$. It is easy to check that Assumptions \ref{assu4} and \ref{assu5} still hold. Then, applying Theorem \ref{the005} to (\ref{eqe01}), we obtain the first-order amplitude equation:
		\begin{align*}
			b_{1}(T)=b_{1}(0)+\int^{T}_{0}(b_{1}(s)-\frac{3}{2\pi}b_{1}^{3}(s))\textup{d}s+\int^{T}_{0} \frac{8\sqrt{2}}{3\pi^{\frac{3}{2}}} b_{1}(s)\textup{d}\tilde{\beta}(s),
		\end{align*}
		and the second-order amplitude equation:
		\begin{align*}
    	b_{2}(T)=\int^{T}_{0}(b_{2}(s)-\frac{9}{2\pi}b_{2}(s)b_{1}^{2}(s))\textup{d}s+\int^{T}_{0}(\bar{\sigma}_{1} b_{2}^2(s)+\bar{\sigma}_{2}b_{2}(s)b^{2}_{1}(s)+\bar{\sigma}_{3}b_{1}^4(s))^{\frac{1}{2}}\textup{d}B(s),
		\end{align*}
	where $\tilde{\beta}(t)$ is the rescaled version of $\beta(t)$, $B(t)$ is one-dimensional real-valued Brownian motion, $b_{1}(0)=\frac{P_{c}u_{0}}{\varepsilon}$, and
	\begin{align*}
		\bar{\sigma}_{1}&=\frac{2^{12}}{\pi^{6}}\sum_{k,j=1}^{\infty}\frac{1}{(k^{2}+j^{2}+k+j   )(2k+1)^2(4k^2+4k-3)^2 (2j+1)^2(4j^2+4j-3)^2}\\
		&~~~+\frac{128}{9\pi^3},\\
		\bar{\sigma}_{2}&=-(h^2-1)(\frac{2}{\pi})^{\frac{5}{2}},~~
		\bar{\sigma}_{3}=\frac{9(h^2-1)^2}{16\pi^2}.
	\end{align*}
	
	Furthermore, if $u_{0}=\varepsilon \sqrt\frac{2}{\pi}\sin(x)$,  then we obtain
	\begin{align}
		u(t)&=\varepsilon b_{1}(\varepsilon^{2}t)\sqrt{\frac{2}{\pi}}\sin(x)+\mathcal{O}(\varepsilon^2), \label{eqe107} \\
		u(t)&=\varepsilon b_{1}(\varepsilon^{2}t)\sqrt{\frac{2}{\pi}}\sin(x)+\varepsilon^{2} b_{2}(\varepsilon^{2}t)\sqrt{\frac{2}{\pi}}\sin(x)+\mathcal{O}(\varepsilon^{\frac{5}{2}}),
		\label{eqe108}
	\end{align}
where (\ref{eqe107}) follows from Theorem 2.1 in \cite{Fu1}, and (\ref{eqe108}) is based on Theorem \ref{the005}.

Next, let us compare error $\bar{R}_{3}(t):=\|u(t)-\varepsilon b_{1}(\varepsilon^{2}t)\sqrt{\frac{2}{\pi}}\sin(x)\|_{L^{2}}$, and error $\bar{R}_{4}(t):=\|u(t)-\varepsilon b_{1}(\varepsilon^{2}t)\sqrt{\frac{2}{\pi}}\sin(x)-\varepsilon^{2} b_{2}(\varepsilon^{2}t)\sqrt{\frac{2}{\pi}}\sin(x)\|_{L^{2}}$ via numerical analysis.

In Figure 2, we consider $\varepsilon=0.01, h=10$ and  $\varepsilon=0.001, h=30 $ respectively, and show the mean  of the error of 100 samples in the sense of $L^2$.

Through the comparisons, it is intuitive to illustrate  that the approximation solutions shown in (\ref{eqe301}) and (\ref{eqe108}) are more accurate. In terms of computational efficiency, the time required to solve the first-order and second-order amplitude equations for each sample is only $\mathcal{O}(10^{-5})$ seconds more than that for the first-order amplitude equation. This means that we can construct the approximate solution by the first-order and second-order amplitude equations for $\varepsilon=0.01$, which maintains the same accuracy as that constructed by the first-order amplitude equation for $\varepsilon=0.001$, without requiring excessive computational time.

\begin{figure}[h]
	\subfigure[$\varepsilon=0.01,h=10$]{
		\hspace{-0.5cm}
		\includegraphics[width=8.5cm]{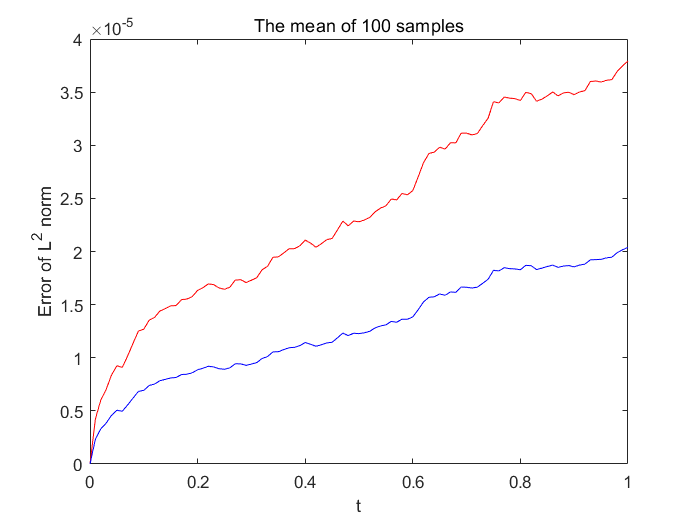}
		\label{eqf005}
	}
	\hspace{0.000005cm}
	\subfigure[$\varepsilon=0.001,h=30$]
	{\centering
		\includegraphics[width=8.5cm]{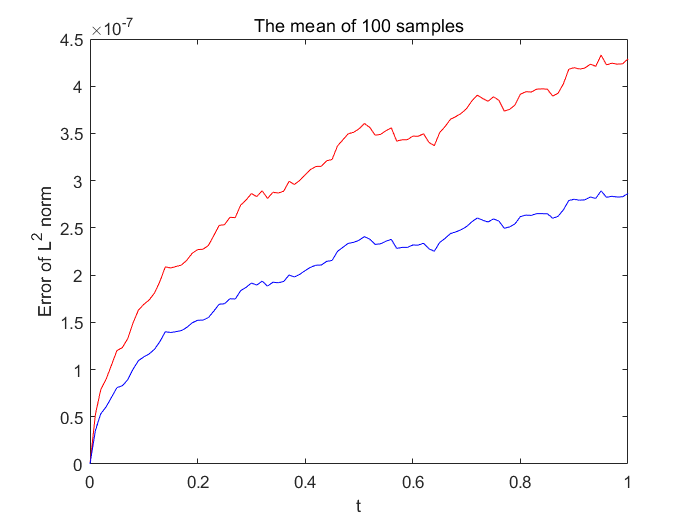}
		\label{eqf006}}
	\caption{Red and blue curves represent the time evolution  of the error $\bar{R}_{3}(t)$ and  $\bar{R}_{4}(t)$, respectively.}
	\label{eqf009}
\end{figure}
		\section{Conclusions}
		In this paper, we  obtain the approximate solutions of a class of SPDEs with multiplicative noise by the first-order and second-order amplitude equations. Compared to the approximate solutions constructed by the first-order amplitude equations \cite{Fu1}, the approximate ones of us  enjoy improved convergence rate.
		And, we use an example to illustrate this viewpoint in Section 5.

		We would like to comment the possible extensions in terms of our results.  We can consider that the diffusion part consists of multiplicative and additive noise at the same time. If the additive noise is  of order $\varepsilon^2$, it is possible to obtain the high-order approximate solutions. However,  it may be a challenge that the additive noise is  of order $\varepsilon$. And, we can attempt to derive the higher-order amplitude equations if we let $u(t)=\varepsilon\psi(\varepsilon^{2}t)+ \varepsilon a_{1}(\varepsilon^{2}t)+\varepsilon^{2} a_{2}(\varepsilon^{2}t)+\varepsilon^{3} a_{3}(\varepsilon^{2}t)+\cdots$. In addition, it is meaningful to consider the high-order approximation of SPDEs with quadratic nonlinearity, which can be used to
		the study of stochastic Burger's equation and stochastic surface growth model.
		We also hope ours result can be applied to random dynamics.
		In \cite{Bl7}, authors investigated the the approximation of random invariant manifolds of SPDEs with linear multiplicative noise via amplitude equations.
		So far, few results focus on  nonlinear multiplicative noise.
		Therefore, we want to  explore the connection  between mean square random invariant manifolds and amplitude equations in future.
		
		\section*{Acknowledgements}
The work is supported in part by the NSFC Grant Nos. 12171084 and  the fundamental Research
Funds for the Central Universities No. RF1028623037.

		\section*{Data Availability Statements}
		Data sharing not applicable to this article as no datasets were generated or analysed during the current study.
		\section*{Conflict of Interest}
		The authors declare that they have no conflict of interest.
		
		

	\end{document}